\documentclass[10pt]{amsart}

\newcounter{mnote}
  \let\oldmarginpar\marginpar
    \renewcommand\marginpar[1]{\-\oldmarginpar[\raggedleft\footnotesize #1]%
    {\raggedright\footnotesize #1}}

\usepackage[english]{babel}

\usepackage{amssymb}
\usepackage[leqno]{amsmath}
\usepackage{mathrsfs}
\usepackage{stmaryrd}
\usepackage{chemarrow}
\usepackage[norelsize]{algorithm2e}
\usepackage{enumerate}
\usepackage{graphicx}
\usepackage[all]{xy}
\usepackage{tikz}
\usetikzlibrary{arrows}
\usepackage{multirow}
\usepackage{tabularx}
\usepackage{booktabs}
\usepackage[colorinlistoftodos]{todonotes}
\usepackage[colorlinks=true, allcolors=blue]{hyperref}
\usepackage[hyperpageref]{backref}
\usepackage{subcaption} 
\numberwithin{equation}{section}
\allowdisplaybreaks

\newtheorem{theorem}{Theorem}[section]
\newtheorem{lemma}[theorem]{Lemma}

\newtheorem{remark}[theorem]{Remark}

\newtheorem{hypothesis}[theorem]{Hypothesis}

\newcommand{\dd}{\,{\rm d}}

\newcommand{\curl}{\operatorname{curl}}
\renewcommand{\div}{\operatorname{div}}
\newcommand{\grad}{\operatorname{grad}}

\begin{document}
\title[Low-order finite element complex]{Low-order finite element complex with application to a fourth-order elliptic singular perturbation problem}
\author{Xuewei Cui}%
\address{School of Mathematics, Shanghai University of Finance and Economics, Shanghai 200433, China}%
\email{xueweicui@stu.sufe.edu.cn }%
\author{Xuehai Huang}%
\address{School of Mathematics, Shanghai University of Finance and Economics, Shanghai 200433, China}%
\email{huang.xuehai@sufe.edu.cn}%

\thanks{
This work was supported by the National Natural Science Foundation of China Project 12171300. 
}
\makeatletter
\@namedef{subjclassname@2020}{\textup{2020} Mathematics Subject Classification}
\makeatother

\begin{abstract}
A low-order nonconforming finite element discretization of a smooth de Rham complex starting from the $H^2$ space in three dimensions is proposed, involving an $H^2$-nonconforming finite element space, a new tangentially continuous $H^1$-nonconforming vector-valued finite element space, the lowest-order Raviart-Thomas space, and piecewise constant functions. While nonconforming for the smooth complex, the discretization conforms to the classical de Rham complex. It is applied to develop a decoupled mixed finite element method for a fourth-order elliptic singular perturbation problem, focusing on the discretization of a generalized singularly perturbed Stokes-type equation.
In contrast to Nitsche's method, which requires additional stabilization to handle boundary layers, the nodal interpolation operator for the lowest-order N\'{e}d\'{e}lec element of the second kind is introduced into the discrete bilinear forms. This modification yields a decoupled mixed method that achieves optimal convergence rates uniformly with respect to the perturbation parameter, even in the presence of strong boundary layers, without requiring any additional stabilization.
\end{abstract}

\maketitle

\section{Introduction}
We aim to construct a nonconforming finite element discretization of the smooth de Rham complex on a bounded domain $\Omega \subset \mathbb{R}^{3}$:
\begin{equation}\label{deRham2}
0\xrightarrow{\subset} H_0^2(\Omega)\xrightarrow{\nabla} H_0^1(\Omega;\mathbb{R}^{3})\xrightarrow{\curl} H_0(\div,\Omega)\xrightarrow{\div}L_0^{2}(\Omega) \xrightarrow{}0,
\end{equation}
where the Sobolev spaces are defined as
\begin{align*}
H(\div,\Omega)&:=\{\boldsymbol{v}\in L^{2}(\Omega;\mathbb{R}^{3})
:\div\boldsymbol{v}\in L^{2}(\Omega)\},\\
H_{0}(\div,\Omega)&:=\{\boldsymbol{v}\in H(\div,\Omega)
:\boldsymbol{v}\cdot\boldsymbol{n}=0\, \,\mathrm{on}\,\, \partial\Omega\},
\end{align*}
and $L_0^{2}(\Omega)$ denotes the subspace of $L^2(\Omega)$ consisting of functions with zero mean value.
Hilbert complexes such as~\eqref{deRham2} play a fundamental role 
in theoretical analysis and the design of stable numerical methods for partial differential equations~\cite{ArnoldFalkWinther2006,ArnoldFalkWinther2010,Arnold2018,ChenHuang2018,ArnoldHu2021}. 
The complex \eqref{deRham2} is smoother than
the standard (domain) de Rham complex 
\begin{equation}\label{deRhamthreedim}
0\xrightarrow{\subset} H_0^1(\Omega)\xrightarrow{\nabla} H_0(\curl,\Omega)\xrightarrow{\curl} H_0(\div,\Omega)\xrightarrow{\div}L_0^{2}(\Omega) \xrightarrow{}0,
\end{equation}
where the relevant Sobolev spaces are defined as
\begin{align*}
H(\curl,\Omega)&:=\{\boldsymbol{v}\in L^{2}(\Omega;\mathbb{R}^{3})
:\curl\boldsymbol{v}\in L^{2}(\Omega;\mathbb{R}^{3})\},\\
H_{0}(\curl,\Omega)&:=\{\boldsymbol{v}\in H(\curl,\Omega)
:\boldsymbol{v}\times\boldsymbol{n}=0 \,\,\mathrm{on}\,\, \partial \Omega\}.
\end{align*}
In addition to~\eqref{deRham2}, other smooth de Rham complexes include the so-called Stokes complexes:
\begin{equation}\label{stokethreedim1}
0\xrightarrow{\subset} H_0^1(\Omega)\xrightarrow{\nabla} H_0(\grad\curl,\Omega)\xrightarrow{\curl} H_0^1(\Omega;\mathbb{R}^{3})\xrightarrow{\div}L_0^{2}(\Omega) \xrightarrow{}0,
\end{equation}
\begin{equation}\label{stokethreedim2}
0\xrightarrow{\subset} H_0^2(\Omega)\xrightarrow{\nabla} H_0^1(\curl,\Omega)\xrightarrow{\curl} H_0^1(\Omega;\mathbb{R}^{3})\xrightarrow{\div}L_0^{2}(\Omega) \xrightarrow{}0,
\end{equation}
with the associated Sobolev spaces given by
\begin{align*}
H_0(\grad\curl,\Omega)&:=\{\boldsymbol{v}\in H_{0}(\curl,\Omega)
:\curl\boldsymbol{v}\in H_0^{1}(\Omega;\mathbb{R}^{3})\}, \\
H_0^1(\curl,\Omega)&:=\{\boldsymbol{v}\in H_0^{1}(\Omega;\mathbb{R}^{3})
:\curl\boldsymbol{v}\in H_0^{1}(\Omega;\mathbb{R}^{3})\}.
\end{align*}

Finite element discretizations of the de Rham complex~\eqref{deRhamthreedim} have been extensively studied in the literature; see, for example, \cite{Hiptmair1999,Hiptmair2001,Arnold2018,ArnoldFalkWinther2006,ArnoldFalkWinther2010,ChenHuang2024,ChenHuang2024a,ChristiansenHuHu2018}. The finite element de Rham complexes with varying degrees of smoothness developed in \cite{ChenHuang2024,ChenHuang2024a} cover the discretizations of all the complexes~\eqref{deRham2}-\eqref{stokethreedim2}.

To reduce the number of degrees of freedom (DoFs), conforming macro-element Stokes complexes on split meshes have been developed for the Stokes complex~\eqref{stokethreedim1}; see \cite{ChristiansenHu2018,HuZhangZhang2022}. A conforming virtual element construction for the same complex was proposed in \cite{BeiraoDassiVacca2020}. In addition, nonconforming discretizations of the Stokes complex~\eqref{stokethreedim1} have been explored in \cite{Huang2023,HuangZhang2024,ZhangZhangZhang2023}.
For the Stokes complex~\eqref{stokethreedim2}, conforming finite element discretizations have been studied in \cite{ChenHuang2024,Neilan2015}, as well as on split meshes in \cite{FuGuzmanNeilan2020,GuzmanLischkeNeilan2022}. Nonconforming discretizations of the Stokes complex~\eqref{stokethreedim2} were investigated in \cite{TaiWinther2006,GuzmanNeilan2012}. In two dimensions, both conforming and nonconforming finite element discretizations of the Stokes complex have been extensively studied; see, for example, \cite{AustinManteuffelMcCormick2004,MardalTaiWinther2002,Lee2010,ChristiansenHu2018,FalkNeilan2013,GuzmanJohnnyNeilan2014,GuzmanLischkeNeilan2020,GuzmanNeilan2012,ChenHuang2022,TongZhaiZhang2025}.

A conforming finite element discretization of the smooth de Rham complex~\eqref{deRham2} was constructed in~\cite{ChenHuang2024}. 
However, to the best of our knowledge, nonconforming finite element discretizations of complex~\eqref{deRham2} have not yet been explored. In this paper, we develop a low-order nonconforming finite element discretization of complex~\eqref{deRham2}:
\begin{equation}\label{homcomplex}
0\xrightarrow{\subset} W_{h}\xrightarrow{\nabla} \Phi_{h}\xrightarrow{\curl} V_{h}^{\div}\xrightarrow{\div}\mathcal{Q}_{h} \xrightarrow{}0.
\end{equation}
The space $W_h$ is taken as the continuous $H^2$-nonconforming finite element space introduced in~\cite{TaiWinther2006}, $V_{h}^{\div}$ is chosen as the lowest-order Raviart-Thomas element space~\cite{RaviartThomas1977,Nedelec1980}, and $\mathcal{Q}_h$ consists of piecewise constant functions. The main focus of this work is the construction of a new tangentially continuous, $H^1$-nonconforming, vector-valued finite element space $\Phi_h$.
The shape function space of $W_h$ is $\mathbb{P}_{2}(T) \oplus b_T\,\mathbb{P}_{1}(T)$, where $b_T=\lambda_0\lambda_1\lambda_2\lambda_3$ is the quartic bubble function on $T$ with $\lambda_i\, (i = 0,1, 2, 3) $ denoting the barycentric coordinate corresponding to the vertex $\texttt{v}_i$. 
Guided by the smooth de Rham complex~\eqref{deRham2}, the shape function space of $\Phi_{h}$ must contain $ \nabla(b_T \,\mathbb{P}_1(T))$ to ensure weak continuity of the normal component across element interfaces. 
To this end, we augment the lowest-order N\'{e}d\'{e}lec element of the second kind~\cite{Nedelec1986} and define the local shape function space as
$$ \Phi(T) := \mathbb{P}_1(T; \mathbb{R}^3)\,\oplus\, \nabla(b_T \,\mathbb{P}_1(T)),$$ 
which has dimension 16. 
The DoFs for $\Phi(T)$ are given by
\begin{equation*}
\int_{e}(\boldsymbol{v}\cdot \boldsymbol{t})\,q\dd s, \quad \forall\, q\in\mathbb{P}_{1}(e), \,e\in \Delta_{1}(T);\qquad
\int_{F}\boldsymbol{v}\cdot \boldsymbol{n} \dd S,\quad F\in \Delta_{2}(T).	
\end{equation*}
The resulting global finite element space $\Phi_h$ has continuous tangential components across element interfaces, i.e., $\Phi_h \subset H_0(\curl,\Omega)$. 
The finite element complex~\eqref{homcomplex} thus constitutes a nonconforming discretization of the smooth de Rham complex~\eqref{deRham2}, while remaining a conforming discretization of the standard de Rham complex~\eqref{deRhamthreedim}.

We then apply the following latter part of the finite element complex~\eqref{homcomplex}:
	\begin{equation}\label{shorthomcomplex}
		\Phi_{h}\xrightarrow{\curl} V_{h}^{\div}\xrightarrow{\div}\mathcal{Q}_{h} \xrightarrow{}0,
	\end{equation}
	together with the quadratic Lagrange element, to develop a decoupled mixed finite element method for the following fourth-order elliptic singular perturbation problem:
\begin{equation}\label{eqtwo}
	\begin{cases}
		\varepsilon^2\Delta^{2}u  -\Delta u=f& \qquad \,\,\mathrm{in}\,\,\Omega,\\
		u=\partial_{n}u=0 & \qquad\,\,\mathrm{on}\,\, \partial\Omega,
	\end{cases}
\end{equation}
where $f \in L^{2}(\Omega)$ is a given function, $\Delta^2$ is the biharmonic operator, $\partial_n$ represents the normal derivative on the boundary, and $\varepsilon>0$ is a real parameter. Notably, as $\varepsilon \to 0$, problem~\eqref{eqtwo} degenerates into the classical Poisson equation, giving rise to boundary layer phenomena.

Classical approaches to problem~\eqref{eqtwo} based on $H^2$-conforming finite elements~\cite{HuLinWu2024,ChenHuang2021,ChenChenGaoHuangEtAl2025,ZhangShangyou2009} have been studied in~\cite{Semper1992}. However, due to the considerable difficulty in implementing $H^2$-conforming elements, $H^2$-nonconforming elements, such as those developed in~\cite{WangWuXie2013,ChenChenXiao2014,WangXuHu2006,ZhangWang2008,XieShiLi2010,WangMeng2007,ChenZhaoShi2005,NilssenTai2001}, are more commonly adopted. An alternative strategy is the $C^0$ interior penalty discontinuous Galerkin method, which employs the Lagrange element spaces and was introduced in~\cite{BrennerNeilan2011,FranzRoos2014}. 
Most of the aforementioned methods, including those in \cite{WangWuXie2013,ChenChenXiao2014,WangXuHu2006,WangMeng2007,ChenZhaoShi2005,NilssenTai2001}, were designed for the primal formulation of the problem~\eqref{eqtwo}. While these methods typically yield uniform error estimates with respect to $\varepsilon$, they do not achieve optimal convergence rates. In particular, as $\varepsilon \to 0$, the convergence rate is generally limited to $\mathcal{O}(h^{1/2})$, which is sharp but suboptimal.

To develop a robust and optimal discretization for the fourth-order elliptic singularly perturbed problem~\eqref{eqtwo} with boundary layers, it is crucial that the scheme reduce to the standard discretization of the Poisson equation when $\varepsilon = 0$. In this limiting case, the discrete space, bilinear forms, and right-hand side should involve only the boundary condition $u = 0$, without enforcing $\partial_{n} u = 0$ strongly or weakly.

A common strategy for achieving this is to impose the boundary condition $\partial_n u = 0$ weakly using Nitsche's method \cite{Nitsche1971} or the penalty technique~\cite{Arnold1982}, instead of incorporating $\partial_n u = 0$ directly into the discrete space, as done in \cite{GuzmaneykekhmanNeilan2012,WangHuangTangZhou2018,HuangShiWang2021}. Another alternative is to adopt the mixed formulation for the fourth-order operator proposed in \cite{HuangTang:2025,LiuHuangWang2020}.	However, Nitsche-type methods introduce additional stabilization terms, and mixed formulations lead to saddle-point systems, making the implementation more involved, even though both approaches can deliver optimal and parameter-uniform error estimates.

To avoid the penalty terms required by Nitsche-type methods and the saddle-point structure of mixed methods, we instead decouple the fourth-order singularly perturbed problem~\eqref{eqtwo} into a sequence of second-order problems and incorporate a suitable interpolation operator into the discrete bilinear forms. This yields a new finite element scheme that achieves optimal and uniform accuracy across all perturbation regimes while maintaining a simpler implementation framework.
Decoupling high-order elliptic equations into Poisson-type and Stokes-type problems, as in the frameworks proposed in \cite{ChenHuang2018,GallistlDietmar2017}, offers an effective strategy for designing efficient numerical methods for such equations.
Applying the framework in \cite{ChenHuang2018}, we decouple the fourth-order elliptic singular perturbation problem~\eqref{eqtwo} into a combination of two Poisson equations and a generalized singularly perturbed Stokes-type equation involving the $\curl$ operator: find $u, w\in H_{0}^{1}(\Omega)$,
$\boldsymbol{\phi}\in H_{0}^{1}(\Omega;\mathbb{R}^{3})$,
$\boldsymbol{p}\in H_0(\div,\Omega)$ and $\lambda\in L_0^2(\Omega)$ such that
\begin{subequations}\label{intro:decoupleA}
\begin{align}
			(\nabla w,\nabla v)&=(f,v), \label{intro:AB}\\
			\varepsilon^2(\nabla\boldsymbol{\phi},\nabla\boldsymbol{\psi})+(\boldsymbol{\phi},\boldsymbol{\psi})+(\curl  \boldsymbol{\psi},\boldsymbol{p})&=(\nabla w,\boldsymbol{\psi}), \label{intro:AC}\\
			(\mu,\div\boldsymbol{p})&=0, \label{intro:AC0}\\
			(\curl  \boldsymbol{\phi},\boldsymbol{q})-(\lambda, \div\boldsymbol{q}) &=0,\label{intro:AD}\\
			(\nabla u,\nabla \chi)&=(\boldsymbol{\phi},\nabla\chi), \label{intro:AE}
	\end{align}
	for any $v,\chi\in H_{0}^{1}(\Omega)$, $\boldsymbol{\psi}\in H_{0}^{1}(\Omega;\mathbb{R}^{3})$,
	$\boldsymbol{q}\in H_0(\div,\Omega)$ and $\mu\in L_0^2(\Omega)$.
\end{subequations}
The well-posedness of the generalized singularly perturbed Stokes-type equation \eqref{intro:AC}-\eqref{intro:AD} is related to the smooth de Rham complex~\eqref{deRham2}.
We establish the well-posedness of the decoupled formulation \eqref{intro:decoupleA} and demonstrate its equivalence to problem~\eqref{eqtwo}.

The formulation \eqref{intro:decoupleA} is termed ``decoupled'' because it can be solved sequentially: one first computes $w$ from \eqref{intro:AB}; then solves for $\boldsymbol{\phi}$, $\boldsymbol{p}$, and $\lambda$ using \eqref{intro:AC}-\eqref{intro:AD}; and finally obtains $u$ from \eqref{intro:AE}.

Based on the decoupled formulation \eqref{intro:decoupleA}, we propose a novel decoupled nonconforming finite element method for the fourth-order elliptic singular perturbation problem~\eqref{eqtwo}, using the short finite element complex~\eqref{shorthomcomplex} in conjunction with the quadratic Lagrange element.
The weak formulations \eqref{intro:AB} and \eqref{intro:AE} of the Poisson equations are discretized using quadratic Lagrange elements, while the finite element spaces $\Phi_{h}$, $V_{h}^{\div}$, and $\mathcal{Q}_h$ from complex~\eqref{homcomplex} are employed to approximate $\boldsymbol{\phi}$, $\boldsymbol{p}$, and $\lambda$ in the generalized singularly perturbed Stokes-type system \eqref{intro:AC}-\eqref{intro:AD}, respectively.

Inspired by the approach in \cite{WangXuHu2006,WangMeng2007}, we introduce the nodal interpolation operator $I_h^{\rm ND}$ associated with the lowest-order N\'{e}d\'{e}lec element of the second kind~\cite{Nedelec1986} into the discrete bilinear forms. Specifically, we replace $(\boldsymbol{\phi}_h, \boldsymbol{\psi})$ and $(\curl \boldsymbol{\phi}_h, \boldsymbol{q})$ with $(I_h^{\rm ND} \boldsymbol{\phi}_h, I_h^{\rm ND} \boldsymbol{\psi})$ and $(\curl (I_h^{\rm ND} \boldsymbol{\phi}_h), \boldsymbol{q})$, respectively.
The resulting decoupled finite element method achieves optimal convergence rates uniformly with respect to the perturbation parameter $\varepsilon$, even in the presence of strong boundary layers, without requiring additional stabilization, compared with the Nitsche-type methods in~\cite{GuzmaneykekhmanNeilan2012,WangHuangTangZhou2018,HuangShiWang2021}. Moreover, the discrete systems arising from these Nitsche-type methods and from the mixed methods in \cite{HuangTang:2025,LiuHuangWang2020} correspond to the original fourth-order problem and therefore have condition numbers of order $\mathcal{O}(h^4)$. In contrast, the proposed decoupled formulation consists solely of second-order problems with $\mathcal{O}(h^2)$ condition numbers, facilitating the design of more efficient solvers. 
	We further show that the proposed method is equivalent to  the following nonconforming finite element method: find $u_{h}^W\in W_{h}$ such that
	\begin{equation*}
		\varepsilon^2 (\nabla_h^2u_h^W, \nabla_h^2v)+(\nabla(I_h^{\grad}u_h^W),\nabla(I_h^{\grad}v))=(f,I_h^{\grad}v)\quad\forall~v\in W_{h},
	\end{equation*}
	where $I_h^{\grad}$ is the nodal interpolation operator of the quadratic Lagrange element.

It is worth noting that the interpolation operator introduced in \cite{WangXuHu2006,WangMeng2007} was originally designed to address the divergence of the Morley-Wang-Xu method for the Poisson equation, rather than to improve its convergence rate; consequently, the uniform convergence achieved there remains suboptimal at $\mathcal{O}(h^{1/2})$. In contrast, we are the first to employ the interpolation operator $I_h^{\rm ND}$ explicitly to attain optimal convergence for the finite element method. This approach can also be extended to other singularly perturbed problems.

The remainder of this paper is organized as follows. Section \ref{chap2} introduces the necessary notation and develops a nonconforming finite element complex along with a corresponding commutative diagram. Section~\ref{DMFE} presents the decoupled formulation of the fourth-order elliptic singular perturbation problem, and establishes its equivalence to the original problem. In Section~\ref{chap5}, we propose an optimal and robust decoupled nonconforming finite element method. Finally, Section~\ref{chap6} provides numerical experiments to demonstrate the effectiveness of the proposed approach.
\section{Low-order nonconforming finite element complex}\label{chap2}
In this section, we propose a low-order nonconforming finite element discretization of the smooth de Rham complex~\eqref{deRham2}, given by
\begin{equation}\label{discretecomplex2}
0\xrightarrow{\subset} W_{h}\xrightarrow{\nabla} \Phi_{h}\xrightarrow{\curl} V_{h}^{\div}\xrightarrow{\div}\mathcal{Q}_{h} \xrightarrow{}0, 
\end{equation}
where $W_h$ is the continuous $H^2$-nonconforming finite element space in \cite{TaiWinther2006}, $\Phi_h$ is a new tangential-continuous $H^1$-nonconforming vector-valued finite element space constructed in this section, $V_{h}^{\div}$ is the lowest-order Raviart-Thomas element space~\cite{RaviartThomas1977,Nedelec1980}, and $\mathcal{Q}_{h}$ is the piecewise constant space. We assume throughout this paper that the bounded domain $\Omega \subset \mathbb{R}^3$ is topologically trivial.

\subsection{Notation}
Given an integer $m \geq 0$ and a bounded domain $D \subset \mathbb{R}^{3}$,
we define $H^{m}(D)$  as the standard Sobolev space of functions on $D$. The corresponding norm and
semi-norm are denoted by $\| \cdot \|_{ m,D}$ and $|\cdot |_{m,D}$, respectively. Set $L^{2}(D) = H^{0}(D)$.
For integer $k\geq0$,
let $\mathbb{P}_{k}(D)$ represent the space of all polynomials in $D$ with the total degree no more than $k$.
Set $\mathbb{P}_{k}(D)=\{0\}$ for $k<0$.
Let $L_{0}^{2}(D)$ be the space of functions in $L^{2}(D)$ with vanishing integral average values. For a space $B(D)$ defined on $D$,
let $B(D; \mathbb{R}^3):=B(D) \otimes \mathbb{R}^3$ be its vector version.
We denote $(\cdot, \cdot)_{D}$
as the usual inner product on $L^{2}(D)$ or $L^{2}(D; \mathbb{R}^3)$. We denote $H_{0}^{m}(D)\,(H_{0}^{m}(D; \mathbb{R}^3))$ as the closure of $C_{0}^{\infty}(D)\,(C_{0}^{\infty}(D; \mathbb{R}^3))$
with respect to the norm $\| \cdot \|_{m,D}$.  
In case $D$ is $\Omega$, we abbreviate $\| \cdot \|_{ m,D}$, $|\cdot|_{m,D}$
and $(\cdot, \cdot)_{D}$  as $\| \cdot \|_{m}$, $|\cdot|_{m}$ and $(\cdot, \cdot)$, respectively. Let $Q_D^k: L^{2}(D) \to \mathbb{P}_{k}(D)$ denote the $L^2$-orthogonal projection operator.
Its vector-valued analogue is still denoted by $Q_D^k$. For brevity, we write $Q_D := Q_D^0$ for the projection onto piecewise constants.   
We use $\boldsymbol{n}_{\partial D}$ to denote the unit outward normal vector of $\partial D$, which will be abbreviated as $\boldsymbol{n}$ if it does not cause any confusion.

Let $\{\mathcal{T}_{h}\}_{h>0}$ be a regular family of tetrahedral meshes of $\Omega\subset \mathbb{R}^{3}$, where $h=\max_{T\in\mathcal{T}_{h}}h_T$ with $h_T$ being the diameter of tetrahedron $T$. For $\ell=0,1,2$,
denote by $\Delta_{\ell}(\mathcal T_h)$ and $\Delta_{\ell}(\mathring{\mathcal T}_h)$ the set of all subsimplices and all interior subsimplices of dimension $\ell$ in the partition $\mathcal{T}_{h}$, respectively. For a tetrahedron $T$, we let $\Delta_{\ell}(T)$ denote the set of subsimplices of dimension $\ell$. 
For a subsimplex $f$ of $\mathcal{T}_{h}$, let $\mathcal T_f$ be the set of all tetrahedrons in $\mathcal{T}_h$ sharing $f$.
Denote by $\omega_{T}$ the union of all the simplices in the set $\{\mathcal T_{\texttt{v}}\}_{\texttt{v}\in\Delta_0(T)}$.
For a face $F\in \Delta_{2}(\mathcal T_h)$, let $\Delta_{1}(F)$ be the set of all edges of $F$, and $\boldsymbol{n}_F$ be its unit normal vector, which will be abbreviated as $\boldsymbol{n}$. 
For an edge $e\in \Delta_{1}(\mathcal T_h)$, let $\boldsymbol{t}_e$ be its unit tangent vector, which will be abbreviated as $\boldsymbol{t}$. 
Consider two adjacent tetrahedrons, $T_1$ and $T_2$ sharing an interior face $F$. Define the jump of a function $w$ on $F$ as
\begin{equation*}
[w]:=(w|_{T_1})|_F\boldsymbol{n}_F\cdot\boldsymbol{n}_{\partial T_1}+(w|_{T_2})|_F\boldsymbol{n}_F\cdot\boldsymbol{n}_{\partial T_2}.
\end{equation*}
On a face $F$ lying on the boundary $\partial\Omega$, the jump becomes $[w]:=w|_F$.
For a vector-valued function $\boldsymbol{v}$, define
\begin{equation*}
\div\boldsymbol{v}=\nabla\cdot\boldsymbol{v},\quad \curl\boldsymbol{v}=\nabla\times\boldsymbol{v},
\end{equation*}
where $\nabla$ denotes the gradient operator.

Let $\nabla_{h}$,  $\curl_{h}$  and $\div_{h}$  be the element-wise counterpart of  $\nabla$,
\,$\curl$ and $\div$ with respect to $\mathcal{T}_{h}$, respectively. For piecewise smooth function $v$, introduce broken norms:
\begin{equation*}
	\|v\|_{1,h}:=(\|v\|_{0}^2+|v|_{1,h}^2)^{1/2},\quad\;\; |v|_{j,h}:=\|\nabla_{h}^jv\|_0\;\;\;  \textrm{ for }j=1,2.
\end{equation*}
In this paper, we use ``$\lesssim \cdots$'' to mean that ``$\leq C \cdots$'',
where $C$ is a generic positive constant independent of the mesh size $h$ and the singular perturbation parameter $\varepsilon$, which may take different values in different contexts. Moreover, $A \eqsim B$ means that $A \lesssim B$ and $B \lesssim A$.

\subsection{Finite element de Rham complex}
We revisit the finite element de Rham complex~\cite{Hiptmair1999,Hiptmair2001,ArnoldFalkWinther2006,ArnoldFalkWinther2010,Arnold2018} in this subsection.

Recall the quadratic Lagrange element space~\cite{CiarletWagschal1971,Zlamal1968}
\begin{equation*}
V_{h}^{\rm grad}=\{v_{h}\in H_0^1(\Omega): v_h|_{T} \in \mathbb{P}_{2}(T) \textrm{ for } T\in \mathcal{T}_{h}\}
\end{equation*}
with the degrees of freedom (DoFs) given by
\begin{subequations}\label{Lagrange2DoFs}
\begin{align}
v(\texttt{v}),&\quad \texttt{v}\in \Delta_{0}(T)\label{Lagrange2DoF1},\\
\int_{e}v \dd s,&\quad e\in \Delta_{1}(T)\label{Lagrange2DoF2},
\end{align}
\end{subequations}
the lowest order N\'{e}d\'{e}lec element space of the second kind~\cite{Nedelec1986}
\begin{equation*}
V_{h}^{\rm ND}=\{\boldsymbol{v}_{h}\in H_0(\curl,\Omega): \boldsymbol{v}_h|_{T} \in \mathbb{P}_{1}(T; \mathbb{R}^{3}) \textrm{ for } T\in \mathcal{T}_{h}\}
\end{equation*}
with the DoFs given by
\begin{equation}\label{VNDDoFs}
\int_{e}(\boldsymbol{v}\cdot \boldsymbol{t})\,q\dd s,\quad \forall\, q\in\mathbb{P}_{1}(e), \,e\in \Delta_{1}(T),
\end{equation}
the lowest order Raviart-Thomas element space~\cite{RaviartThomas1977,Nedelec1980}
 \begin{equation*}
 V_{h}^{\div}=\{\boldsymbol{v}_{h}\in H_0(\div,\Omega)
 :\boldsymbol{v}_h|_{T} \in \mathbb{P}_{0}(T;\mathbb{R}^{3}) \,\oplus\,\boldsymbol{x} \mathbb{P}_{0}(T)\textrm{ for } T\in \mathcal{T}_{h}\}
 \end{equation*}
with DoFs given by
 \begin{equation*}
 \int_F\boldsymbol{v}\cdot\boldsymbol{n}\dd S,\quad F\in \Delta_{2}(T),
 \end{equation*}
and the scalar-valued piecewise constant space
\begin{equation*}
\mathcal{Q}_{h}=\{v_{h}\in L_0^2(\Omega) :v_h|_{T} \in\mathbb{P}_{0}(T)\textrm{ for } T\in \mathcal{T}_{h}\}.
 \end{equation*}
A combination of the above finite element spaces leads to the following finite element de Rham complex~\cite{Hiptmair1999,Hiptmair2001,ArnoldFalkWinther2006,ArnoldFalkWinther2010,Arnold2018}:
\begin{equation}\label{discretecomplex3}
0\xrightarrow{\subset} V_{h}^{\rm grad}\xrightarrow{\nabla}V_{h}^{\rm ND}\xrightarrow{\curl} V_{h}^{\div}\xrightarrow{\div} \mathcal{Q}_{h} \xrightarrow{}0.
\end{equation}

Let
$\Pi^{\grad}_h: H_0^1(\Omega) \to V_{h}^{\rm grad},\, \Pi_h^{\rm ND}: H_0(\curl,\Omega)\to V_{h}^{\rm ND},\, 
I_h^{\rm RT}: H_0(\div,\Omega)\to V_{h}^{\div}$
and $Q_{h}: L_0^{2}(\Omega)\rightarrow  \mathcal{Q}_{h}$ be the $L^2$-bounded projection operators devised in~\cite[(5.2)]{ArnoldGuzman2021} and \cite[Page 826]{ChristiansenWinther2008}, which admit the following commutative diagram:
\begin{equation}\label{dcommutesdeRham}
\resizebox{0.8\hsize}{!}{$
\begin{array}{c}
\xymatrix{
0\ar[r]^-{} &H_0^{1}(\Omega)\ar[r]^-{\nabla} \ar[d]^{\Pi^{\grad}_h}&  H_0(\curl,\Omega) \ar[r]^-{\curl} \ar[d]^{\Pi^{\rm ND}_h}&  H_0(\div,\Omega) \ar[r]^-{\div} \ar[d]^{I^{\rm RT}_h}&  L_0^{2}(\Omega) \ar[r]^-{} \ar[d]^{Q_h}& 0
\\
0\ar[r]^-{} & V_{h}^{\rm grad}\ar[r]^-{\nabla} & V_{h}^{\rm ND}\ar[r]^-{\curl} &V_{h}^{\div}
\ar[r]^-{\div } & \mathcal{Q}_{h} \ar[r]^-{} & 0
}
\end{array}
$}.
\end{equation}
We have the following interpolation error estimates
\begin{equation}
\label{eq:IhND2}
\|\boldsymbol{v}-\Pi_h^{\rm ND}\boldsymbol{v}\|_{0,T} + h_{T}|\boldsymbol{v}-\Pi_h^{\rm ND}\boldsymbol{v}|_{1,T}\lesssim h_{T}^{s}|\boldsymbol{v}|_{s,\omega_T} 
\end{equation}
for any $\boldsymbol{v}\in H^s(\Omega;\mathbb R^3)\cap H_0(\curl,\Omega)$ and $1\leq s\leq 2$,
and
\begin{equation}
\label{eq:Ihdivprop2}
\|\boldsymbol{v}-I_h^{\rm RT}\boldsymbol{v}\|_{0}
\lesssim h|\boldsymbol{v}|_{1}\qquad\forall~\boldsymbol{v}\in H^1(\Omega;\mathbb R^3)\cap H_0(\div,\Omega).
\end{equation}


\subsection{$H^1$-nonconforming finite element for vectors}
In this subsection, we construct a tangential continuous $H^1$-nonconforming finite element by enriching the linear  N\'{e}d\'{e}lec elements of the second kind~\cite{Nedelec1986} to enhance the weak continuity of the normal component across element interfaces. 

To this end, for each tetrahedron $T$ with vertices $\texttt{v}_0, \texttt{v}_1, \texttt{v}_2, \texttt{v}_3$, we take the shape function space as
\begin{align*}
	\Phi(T) :=\mathbb{P}_{1}(T; \mathbb{R}^{3})\,\oplus\, \nabla(b_T\,\mathbb{P}_{1}(T)),
\end{align*}
where $b_T=\lambda_0\lambda_1\lambda_2\lambda_3$ is the quartic bubble function on $T$ with $\lambda_i\, (i = 0,1, 2, 3) $ denoting the barycentric coordinate corresponding to the vertex $\texttt{v}_i$. 
The DoFs of $\Phi(T)$ are given by
\begin{subequations}\label{PhiTDoFs}
	\begin{align}
		\int_{e}(\boldsymbol{v}\cdot \boldsymbol{t})\,q\dd s,& \quad \forall\, q\in\mathbb{P}_{1}(e), \,e\in \Delta_{1}(T),\label{Phi01}\\
		\int_{F}\boldsymbol{v}\cdot \boldsymbol{n} \dd S,&\quad F\in \Delta_{2}(T). \label{Phi02}
	\end{align}
\end{subequations}
The edge-based DoF \eqref{Phi01} is exactly that of the linear  N\'{e}d\'{e}lec elements of the second kind~\cite{Nedelec1986}, while the face-based DoF \eqref{Phi02} enhances the weak continuity of the normal component across element interfaces. The DoFs \eqref{PhiTDoFs} are illustrated in Fig.~\ref{fig:DofPhihWh} (a).

\begin{lemma}\label{lemPhiT}
	The DoFs \eqref{PhiTDoFs} are unisolvent for the shape function space $\Phi(T)$.
\end{lemma}
\begin{proof}
	Both the number of DoFs \eqref{PhiTDoFs} and $\dim \Phi(T)$ are 
	\begin{equation*}
		\dim \mathbb{P}_{1}(T; \mathbb{R}^{3})+\dim\nabla(b_T\,\mathbb{P}_{1}(T))\,=\,16.
	\end{equation*}

	Let $\boldsymbol{v} =\boldsymbol{v}_1+ \nabla(b_T \,p)\in \Phi(T)$ with $\boldsymbol{v}_1\in  \mathbb{P}_{1}(T; \mathbb{R}^{3})$ and $p \in \mathbb{P}_{1}(T)$ belong to the kernel of the DoFs  \eqref{PhiTDoFs}.
	First, we have 
	\begin{equation*}
		\int_{e}(\nabla (b_T\,p)\cdot \boldsymbol{t})\,q\dd s=0  \quad \forall\, q\in\mathbb{P}_{1}(e), \,e\in \Delta_{1}(T).
	\end{equation*}
	Then the vanishing DoF \eqref{Phi01} implies that
	\begin{equation*}
		\int_{e}(\boldsymbol{v}\cdot \boldsymbol{t})\, q\dd s =\int_{e}(\boldsymbol{v}_1\cdot \boldsymbol{t})\, q\dd s=0\quad \forall\, q\in\mathbb{P}_{1}(e), \,e\in \Delta_{1}(T).
	\end{equation*}
	Therefore, $\boldsymbol{v}_1=0$ holds from the unisolvence for the second-kind N\'{e}d\'{e}lec element~\cite[Theorem 5]{Nedelec1986}. Consequently, $\boldsymbol{v}$ can be expressed as 
	$\boldsymbol{v} =\nabla(b_T\,p)$. Furthermore, for face $F_i\in \Delta_{2}(T)$ opposite to vertex $\texttt{v}_i$ ($i=0,1,2,3$), the vanishing DoF \eqref{Phi02} yields
	\begin{equation*}
		\int_{F_i}\nabla (b_T\,p)\cdot \boldsymbol{n}\dd S=(\nabla \lambda_i\cdot \boldsymbol{n})\int_{F_i}b_{F_i}\, p \dd S=0,
	\end{equation*}
	where $b_{F_i}$ is the face bubble function on $F_i$ determined by $b_T=\lambda_ib_{F_i}$.
	This implies that
	\begin{equation*}
		\int_{F_i}b_{F_i}\, p \dd S=0,\quad i=0,1,2,3.
	\end{equation*}
	Expand $p=\sum_{i=0}^3C_i \lambda_{i}$ with $C_i\in\mathbb R$. By a direct computation, we have
	\begin{equation*}
		C_1+C_2+C_3=C_0+C_2+C_3=C_0+C_1+C_3=C_0+C_1+C_2=0.
	\end{equation*}
	Thus, $C_0=C_1=C_2=C_3=0$, and $\boldsymbol{v} = 0$. 
\end{proof}

The $H^1$-nonconforming vector-valued space $\Phi_h$ is given by
\begin{align*}
	\Phi_h=\{&\boldsymbol{v}_{h}\in L^2 (\Omega; \mathbb{R}^{3})
	:\boldsymbol{v}_h|_{T} \in \Phi(T)\textrm{ for } T\in \mathcal{T}_{h}, \text{ all the DoFs \eqref{PhiTDoFs} are} \\
	&\qquad\qquad\qquad\quad\;\;\text{single-valued, and all the DoFs \eqref{PhiTDoFs} on $\partial\Omega$ vanish}\}.
\end{align*}
This construction ensures that $\Phi_h$  has the globally continuous tangential component, i.e., $\Phi_h \subset H_0(\curl,\Omega)$. Moreover, $\Phi_h$ satisfies the weak continuity 
\begin{equation}\label{Phiweakcontinuity1}
	\int_F[\boldsymbol{v}_h]\dd S=0\quad\forall~\boldsymbol{v}_h\in \Phi_h, \,F\in\Delta_{2}(\mathcal T_h).
\end{equation}

\begin{figure}[htbp]
	\centering
	\begin{subfigure}[b]{0.48\textwidth} 
		\centering
		\includegraphics[width=3.8cm]{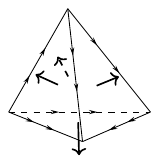} 
		\caption{Sixteen DoFs \eqref{PhiTDoFs} of $\Phi_h$.}
		\label{fig:DofPhih} 
	\end{subfigure}
	\hfill 
	\begin{subfigure}[b]{0.48\textwidth} 
		\centering
		\includegraphics[width=3.8cm]{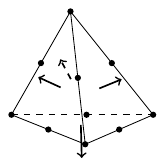}
		\caption{Fourteen DoFs \eqref{WTDoFs} of $W_h$.} 
		\label{fig:DofWh}
	\end{subfigure}
	\caption{DoFs of $\Phi_h$ and $W_h$ on a tetrahedron.}
	\label{fig:DofPhihWh}
\end{figure}

\subsection{Nonconforming finite element complex}
We first revisit the continuous $H^2$-nonconforming finite element space in \cite[Section 2]{TaiWinther2006}.
The shape function space is $\mathbb{P}_{2}(T) \oplus b_T\,\mathbb{P}_{1}(T)$,
and the DoFs are given by (see Fig.~\ref{fig:DofPhihWh} (b)):
\begin{subequations}\label{WTDoFs}
	\begin{align}
		v(\texttt{v}),&\quad \texttt{v}\in \Delta_{0}(T)\label{w1},\\
		\int_{e}v \dd s,&\quad e\in \Delta_{1}(T)\label{w2},\\
		\int_{F}\partial_nv\dd S,&\quad F\in \Delta_{2}(T).\label{w3}
	\end{align}
\end{subequations}
The global $H^2$-nonconforming finite element space is defined by
\begin{equation}\label{Wh}
	\begin{aligned}
		W_h=\{&v_{h}\in L^2(\Omega): v_h|_{T} \in\mathbb{P}_{2}(T) \,\oplus\,  b_T\,\mathbb{P}_{1}(T)\textrm{ for } T\in \mathcal{T}_{h}, \text{all the }\\
		&\text{DoFs in \eqref{WTDoFs} are single-valued, and those located on $\partial\Omega$ vanish}\}.
	\end{aligned}
\end{equation}
It holds that $W_h\subset H_0^1(\Omega)$. 

\begin{theorem}\label{discretecomplex}
Assume that the domain $\Omega$ is contractible.
	The finite element complex \eqref{discretecomplex2} is exact.
\end{theorem}
\begin{proof}
	By the definitions of $W_{h}$ and $\Phi_h$, it is clear that \eqref{discretecomplex2} is a complex. Then we show that the complex is exact.
	
	Take $\boldsymbol{\phi}_h \in \Phi_h\cap \ker(\curl)$. Since $H_0(\curl,\Omega)\cap\ker(\curl)=\nabla H_0^1(\Omega)$~(cf. \cite[Theorem 1.1]{CostabelMcIntosh2010} and \cite[(2.14)]{ArnoldFalkWinther2006}) and $\Phi_h \subset H_0(\curl,\Omega)$, there exists a $w_h\in H_0^1(\Omega)$ such that  $\boldsymbol{\phi}_h= \nabla w_h$ and $w_h|_T \in \mathbb{P}_{2}(T) \,\oplus\,  b_T\,\mathbb{P}_{1}(T)$.
	Apply the weak continuity \eqref{Phiweakcontinuity1} to have $w_h\in W_h$.
	That is $\Phi_h\cap\ker(\curl)=\nabla W_h$. This also means that 
	\begin{align*}
		\dim\curl \Phi_h \,&=\,\dim \Phi_h -\dim W_h = \,|\Delta_{1}(\mathring{\mathcal T}_h)|- |\Delta_{0}(\mathring{\mathcal T}_h)|.
	\end{align*}
	
	It follows from the exactness of the finite element de Rham complex~\cite[Section 5.5]{ArnoldFalkWinther2006} that 
	$\div V_{h}^{\div}=\mathcal{Q}_{h}$, and
	\begin{equation*}
		\dim V_{h}^{\div}\cap\ker(\div)
		=\dim V_{h}^{\div}-\dim \mathcal{Q}_{h}=|\Delta_{2}(\mathring{\mathcal T}_h)|-|\mathcal{T}_{h}| +1.
	\end{equation*}
	Employing the Euler's formula, we have
	\begin{equation*}
		\dim V_{h}^{\div}\cap\ker(\div)-\dim\curl \Phi_h=|\Delta_{0}(\mathring{\mathcal T}_h)|-|\Delta_{1}(\mathring{\mathcal T}_h)|+|\Delta_{2}(\mathring{\mathcal T}_h)|-|\mathcal{T}_{h}| +1=0.
	\end{equation*}
	Therefore, $V_{h}^{\div}\cap\ker(\div)=\curl \Phi_h$, which ends the proof.
\end{proof}

\subsection{Interpolation operators and commutative diagram}
Based on the DoFs given in~\eqref{PhiTDoFs}, we define the interpolation operator $I_h^{\Phi}: H_0^1(\Omega;\mathbb{R}^{3})\rightarrow \Phi_h$ as follows:
\begin{subequations}\label{Iphiinterpolations}
	\begin{align}
		\int_{e}(I_h^{\Phi}\boldsymbol{v})\cdot \boldsymbol{t}\,q\dd s&=\int_{e}(\Pi_h^{\rm ND}\boldsymbol{v})\cdot \boldsymbol{t}\,q\dd s,\quad \forall\, q\in\mathbb{P}_{1}(e), \,e\in \Delta_{1}(\mathring{\mathcal T}_h)\label{Phi1},\\
		\int_{F} (I_h^{\Phi}\boldsymbol{v})\cdot \boldsymbol{n} \dd S&=\int_{F} \boldsymbol{v}\cdot \boldsymbol{n} \dd S ,\qquad\qquad F\in \Delta_{2}(\mathring{\mathcal T}_h) .\label{Phi2}
	\end{align}
\end{subequations}
Based on the DoFs \eqref{WTDoFs}, we define the interpolation operator $I_h^{\rm W}: H_0^2(\Omega) \to W_h$ as follows:
\begin{subequations}\label{Ihw}
	\begin{align}
		(I_h^{\rm W}v)\,(\texttt{v})&=(\Pi^{\grad}_h v)\,(\texttt{v}),\quad \texttt{v}\in \Delta_{0}(\mathring{\mathcal T}_h)\label{Ihw1},\\
		\int_{e} I_h^{\rm W}v \dd s&=\int_{e}\Pi^{\grad}_h v \dd s,\quad e\in \Delta_{1}(\mathring{\mathcal T}_h)\label{Ihw2},\\
		\int_{F}\partial_n{(I_h^{\rm W}v )}\dd S&=\int_{F}\partial_nv\dd S ,\quad\;\;\; F\in \Delta_{2}(\mathring{\mathcal T}_h).\label{Ihw3}
	\end{align}
\end{subequations}
\begin{lemma}\label{Ihphi}
	For $\boldsymbol{v}\in H_0^{1}(\Omega;\mathbb{R}^{3}) \cap H^{s}(\Omega;\mathbb{R}^{3})$ with $1\leq s\leq 2$. We have for $T\in\mathcal T_h$ that 
	\begin{equation}\label{BEqs}
		\|\boldsymbol{v}-I_h^{\Phi}\boldsymbol{v}\|_{0,T} + h_{T}|\boldsymbol{v}-I_h^{\Phi}\boldsymbol{v}|_{1,T}\lesssim h_{T}^{s}|\boldsymbol{v}|_{s,\omega_T}.
	\end{equation}
\end{lemma} 
\begin{proof}
	By an inverse inequality~\cite[Lemma 4.5.3]{BrennerScott2008}, and a scaling argument~\cite[Section 3.1]{Ciarlet1978}, we have
	\begin{align*}
		&\quad\;\,  \|\Pi_h^{\rm ND}\boldsymbol{v}-I_h^{\Phi}\boldsymbol{v}\|_{0,T}^2 + h_T^2|\Pi_h^{\rm ND}\boldsymbol{v}-I_h^{\Phi}\boldsymbol{v}|_{1,T}^2 \\
		&\lesssim h_T \sum_{F\in\Delta_2(T)}\|Q_{F}^0((\Pi_h^{\rm ND}\boldsymbol{v}-I_h^{\Phi}\boldsymbol{v})\cdot \boldsymbol{n})\|_{0,F}^2 \leq h_T \|\boldsymbol{v}-\Pi_h^{\rm ND}\boldsymbol{v}\|_{0,\partial T}^2 ,
	\end{align*}
	which together with a trace inequality~\cite[Theorem 1.6.6]{BrennerScott2008} implies
	\begin{equation*}
		\|\boldsymbol{v}-I_h^{\Phi}\boldsymbol{v}\|_{0,T} + h_T|\boldsymbol{v}-I_h^{\Phi}\boldsymbol{v}|_{1,T}\lesssim \|\boldsymbol{v}-\Pi_h^{\rm ND}\boldsymbol{v}\|_{0,T} + h_T|\boldsymbol{v}-\Pi_h^{\rm ND}\boldsymbol{v}|_{1,T}.
	\end{equation*}
	Hence, \eqref{BEqs} follows from the last inequality and \eqref{eq:IhND2}.
\end{proof}

We can obtain the following commutative diagram.
\begin{lemma}
	It holds the commutative diagram:
	\begin{equation}\label{dcommutes2}
		\resizebox{0.8\hsize}{!}{$
			\begin{array}{c}
				\xymatrix{
					0\ar[r]^-{} &H_0^{2}(\Omega)\ar[r]^-{\nabla} \ar[d]^{I^W_h}&  H_0^{1}(\Omega;\mathbb{R}^{3}) \ar[r]^-{\curl} \ar[d]^{I^{\Phi}_h}&  H_0(\div,\Omega) \ar[r]^-{\div} \ar[d]^{I^{\rm RT}_h}&  L_0^{2}(\Omega) \ar[r]^-{} \ar[d]^{Q_h}& 0\\
					0\ar[r]^-{} &W_{h}\ar[r]^-{\nabla} & \Phi_{h}\ar[r]^-{\curl} & V_{h}^{\div}\ar[r]^-{\div } & \mathcal{Q}_{h} \ar[r]^-{} & 0
				}
			\end{array}
			$}.
	\end{equation}
\end{lemma}
\begin{proof}
	Thanks to the diagram \eqref{dcommutesdeRham}, it suffices to prove
	\begin{align}\label{interpolation01}
		\nabla(I^W_h v) &= I^{\Phi}_h( \nabla v)\qquad\;\;\forall~v\in H_0^{2}(\Omega),\\
		\curl(I^{\Phi}_h\boldsymbol{v}) &=I^{\rm RT}_h(\curl \boldsymbol{v})\quad\forall\,\boldsymbol{v}\in H_0^{1}(\Omega;\mathbb{R}^{3}).\label{interpolation02}
	\end{align}
	
	For $v \in H_0^2(\Omega)$, by \eqref{Phi1}, $\nabla(\Pi^{\grad}_h v)=\Pi^{\rm ND}_h( \nabla v)$ from the diagram \eqref{dcommutesdeRham}, and \eqref{Ihw1}-\eqref{Ihw2}, we have
	\begin{align*}
		\int_{e}( \nabla(I^W_h v) -I^{\Phi}_h( \nabla v))\cdot \boldsymbol{t}\,q\dd s
		&=\int_{e}( \nabla(I^W_h v) -\Pi^{\rm ND}_h( \nabla v))\cdot \boldsymbol{t}\,q\dd s\\
		&=\int_{e}  \nabla(I^W_h v- \Pi^{\grad}_h v)\cdot \boldsymbol{t}\,q\dd s=0
	\end{align*}
	for any $q\in\mathbb{P}_{1}(e)$ and $e \in \Delta_1(\mathring{\mathcal T}_h)$.
	Similarly, we use \eqref{Phi2} and \eqref{Ihw3} to obtain
	\begin{align*}
		\int_{F}( \nabla(I^W_h v)-I^{\Phi}_h( \nabla v))\cdot \boldsymbol{n} \dd S
		=\int_{F}  \nabla(I^W_h v- v)\cdot \boldsymbol{n} \dd S =0,\quad \forall\, F\in \Delta_{2}(\mathring{\mathcal T}_h).
	\end{align*}
	Hence, all the DoFs \eqref{PhiTDoFs} of function $\nabla(I_h^{\rm W} v) - I_h^{\Phi}(\nabla v)\in\Phi_h$ vanish, that is~\eqref{interpolation01} holds.
	
	Next, for any $\boldsymbol{v} \in H_0^1(\Omega; \mathbb{R}^3)$, using the commutative diagram \eqref{dcommutesdeRham}, we have $I^{\rm RT}_h(\curl \boldsymbol{v})=\curl(\Pi^{\rm ND}_h \boldsymbol{v})$. Then for each $ F \in \Delta_2(\mathcal T_h)$, we get from the integration by parts on face $F$ and \eqref{Phi2} that
	\begin{align*}
		\int_{F} (\curl(I^{\Phi}_h\boldsymbol{v}) -I^{\rm RT}_h(\curl \boldsymbol{v}) ) \cdot \boldsymbol{n} \dd S
		=\int_{F} \curl(I^{\Phi}_h\boldsymbol{v} -\Pi^{\rm ND}_h \boldsymbol{v})\cdot \boldsymbol{n} \dd S = 0.
	\end{align*}
	This together with the fact $\curl(I_h^{\Phi} \boldsymbol{v}) - I_h^{\rm RT}(\curl \boldsymbol{v}) \in V_{h}^{\div}$ yields \eqref{interpolation02}.
\end{proof}

For constructing an optimal decoupled method for the fourth-order elliptic singular perturbation problem, we need to introduce the nodal interpolation operator for the finite element space $V_{h}^{\rm ND}$.
Let $I_h^{\rm ND}: H_0^{2}(\Omega;\mathbb R^3)+\Phi_{h}\to V_{h}^{\rm ND}$ be the nodal interpolation operator based on DoFs~\eqref{VNDDoFs}:
\begin{equation}\label{VNDinterp}
	\int_{e}((I_h^{\rm ND}\boldsymbol{v})\cdot \boldsymbol{t})\,q\dd s=\int_{e}(\boldsymbol{v}\cdot \boldsymbol{t})\,q\dd s,\quad \forall\, q\in\mathbb{P}_{1}(e), \,e\in \Delta_{1}(\mathring{\mathcal{T}}_h),
\end{equation}
and $I_h^{\grad}:H_0^{2}(\Omega)+W_{h}\to V_{h}^{\rm grad}$ be the nodal interpolation operator based on the DoFs~\eqref{Lagrange2DoFs}.
It can be verified that 
\begin{align}\label{INDIPhiPiND}
	I_h^{\rm ND}(I_h^{\Phi}\boldsymbol{v}) &=\Pi_h^{\rm ND}\boldsymbol{v} \quad\;\;\forall~\boldsymbol{v}\in H_0(\curl,\Omega), \\
	\notag
	I_h^{\grad}(I_h^{W}v) &=\Pi_h^{\grad}v \quad\;\forall~v\in H_0^{1}(\Omega).
\end{align}
Then
the commutative diagram \eqref{dcommutes2} can be extended to the following three-line commutative diagram
\begin{equation}\label{dcommutesthree}
	\resizebox{0.8\hsize}{!}{$
		\begin{array}{c}
			\xymatrix{
				0\ar[r]^-{} &H_0^{2}(\Omega)\ar[r]^-{\nabla} \ar[d]^{I^W_h}&  H_0^{1}(\Omega;\mathbb{R}^{3}) \ar[r]^-{\curl} \ar[d]^{I^{\Phi}_h}&  H_0(\div,\Omega) \ar[r]^-{\div} \ar[d]^{I^{\rm RT}_h}&  L_0^{2}(\Omega) \ar[r]^-{} \ar[d]^{Q_h}& 0
				\\
				0\ar[r]^-{} &W_{h}\ar[r]^-{\nabla} \ar[d]^{I_h^{\grad}}& \Phi_{h} \ar[r]^-{\curl} \ar[d]^{I^{\rm ND}_h}&  V_{h}^{\div}\ar[r]^-{\div} \ar[d]^{I}& \mathcal{Q}_{h} \ar[r]^-{} \ar[d]^{I}& 0\\
				0\ar[r]^-{} & V_{h}^{\rm grad}\ar[r]^-{\nabla} & V_{h}^{\rm ND}\ar[r]^-{\curl} &V_{h}^{\div}
				\ar[r]^-{\div } &\mathcal{Q}_{h} \ar[r]^-{} & 0
			}
		\end{array}
		$}.
\end{equation}

Applying the standard interpolation error analysis \cite[Theorem 3.1.4]{Ciarlet1978} and \cite[Proposition 2.2.1]{BoffiBrezziFortin2013} together with an inverse inequality \cite[Lemma 4.5.3]{BrennerScott2008}, we have the following estimate for $I_h^{\rm ND}$.
\begin{lemma}
	For $\boldsymbol{v}\in \Phi_{h}$ with $0\leq j\leq 2$. We have that 
	\begin{equation}
		\|\boldsymbol{v}-I_h^{\rm ND}\boldsymbol{v}\|_{0}
		\lesssim h^{j}|\boldsymbol{v}|_{j,h}.\label{eq:IhNDestimes}
	\end{equation}
\end{lemma}

	\section{Decoupled formulation of the fourth-order elliptic singular perturbation problem}\label{DMFE}
In this section, we present  a decoupled variational formulation for the fourth-order elliptic singular perturbation problem \eqref{eqtwo} in three dimensions, and use the smooth de Rham complex~\eqref{deRham2} to establish the well-posedness of the decoupled formulation.

\subsection{The decoupled variational formulation and equivalence}
The primal formulation of  problem \eqref{eqtwo} is to find $u \in H_{0}^{2}(\Omega)$ such that
\begin{align}\label{AA0}
	\varepsilon^2(\nabla^{2} u, \nabla^{2} v)+(\nabla u, \nabla v)=(f, v) \quad \forall\,v \in H_{0}^{2}(\Omega).
\end{align}

Applying the framework in \cite{ChenHuang2018}, the formulation \eqref{AA0} can be equivalently decoupled as follows: find $ w\in H_{0}^{1}(\Omega)$,
$\boldsymbol{\phi}\in H_{0}^{1}(\Omega;\mathbb{R}^{3})$,
$\boldsymbol{p}\in L^{2}(\Omega;\mathbb{R}^{3})/\nabla H^{1}(\Omega)$ and
$u\in H_{0}^{1}(\Omega)$ such that
\begin{align*}
	(\nabla w,\nabla v)&=(f,v) &&\forall \,\, v\in H_{0}^{1}(\Omega),\\
	\varepsilon^2(\nabla\boldsymbol{\phi},\nabla\boldsymbol{\psi})+(\boldsymbol{\phi},\boldsymbol{\psi})+(\curl\boldsymbol{\psi}, \boldsymbol{p})&=(\nabla w,\boldsymbol{\psi})
	&&\forall \,\, \boldsymbol{\psi}\in  H_{0}^{1}(\Omega;\mathbb{R}^{3}),\\
	(\curl\boldsymbol{\phi}, \boldsymbol{q}) &=0  &&\forall \,\, \boldsymbol{q} \in L^{2}(\Omega;\mathbb{R}^{3})/\nabla H^{1}(\Omega),\\
	(\nabla u,\nabla \chi)&=(\boldsymbol{\phi},\nabla\chi) &&\forall \,\, \chi\in H_{0}^{1}(\Omega).
\end{align*}
This decoupled formulation is new. Related decouplings for the fourth-order elliptic singular perturbation problem in two dimensions and the biharmonic equation in three dimensions are discussed in \cite[Section 3.2]{ChenHuang2018}.

To eliminate the quotient space $L^2(\Omega; \mathbb{R}^3)/\nabla H^1(\Omega)$, we introduce a Lagrange multiplier, yielding the following equivalent formulation: find $u, w\in H_{0}^{1}(\Omega)$,
$\boldsymbol{\phi}\in H_{0}^{1}(\Omega;\mathbb{R}^{3})$,
$\boldsymbol{p}\in H_0(\div,\Omega)$ and $\lambda\in L_0^2(\Omega)$ such that
\begin{subequations}\label{decoupleA}
	\begin{align}
		(\nabla w,\nabla v)&=(f,v), \label{AB}\\
		\varepsilon^2(\nabla\boldsymbol{\phi},\nabla\boldsymbol{\psi})+(\boldsymbol{\phi},\boldsymbol{\psi})+(\curl  \boldsymbol{\psi},\boldsymbol{p})&=(\nabla w,\boldsymbol{\psi}), \label{AC}\\
		(\mu,\div\boldsymbol{p})&=0, \label{AC0}\\
		(\curl  \boldsymbol{\phi},\boldsymbol{q})-(\lambda, \div\boldsymbol{q}) &=0,\label{AD}\\
		(\nabla u,\nabla \chi)&=(\boldsymbol{\phi},\nabla\chi), \label{AE}
	\end{align}
	for any $v,\chi\in H_{0}^{1}(\Omega)$, $\boldsymbol{\psi}\in H_{0}^{1}(\Omega;\mathbb{R}^{3})$,
	$\boldsymbol{q}\in H_0(\div,\Omega)$ and $\mu\in L_0^2(\Omega)$.
\end{subequations}

\begin{remark}\rm
	Although the mixed formulation \eqref{decoupleA} is written as a single system, the equations \eqref{AB}-\eqref{AE} are in fact decoupled. In practice, it can be solved sequentially: one first computes $w$ from \eqref{AB}; then solves for $\boldsymbol{\phi}$, $\boldsymbol{p}$, and $\lambda$ using \eqref{AC}-\eqref{AD}; and finally obtains $u$ from \eqref{AE}.
\end{remark}

Next, we analyze the well-posedness of the decoupled formulation \eqref{decoupleA} and its equivalence to the primal formulation \eqref{AA0}. Both \eqref{AB} and \eqref{AE} are Poisson equations, which are evidently well-posed. We now focus on the well-posedness of the generalized singularly perturbed Stokes-type equation \eqref{AC}-\eqref{AD} involving the $\curl$ operator. To facilitate the analysis, we introduce the following bilinear forms
\begin{equation*}
	a(\boldsymbol{\phi},\lambda;\boldsymbol{\psi},\mu):=\varepsilon^2(\nabla\boldsymbol{\phi},\nabla\boldsymbol{\psi})+(\boldsymbol{\phi},\boldsymbol{\psi}),\quad b(\boldsymbol{\psi}, \mu;\boldsymbol{q}):=(\curl \boldsymbol{\psi}, \boldsymbol{q})
	-(\mu,\div\boldsymbol{q}),
\end{equation*}
and
equip space $H_0^1(\Omega; \mathbb{R}^3)$ with the $\varepsilon$-weighted norm
\begin{equation*}
	\|\boldsymbol{\psi}\|_{\varepsilon}:=(\varepsilon^2 |\boldsymbol{\psi}|_1^2+\|\curl\boldsymbol{\psi}\|_0^2+\|\boldsymbol{\psi}\|_0^2)^{1/2}.
\end{equation*}
Clearly, we have the following continuity conditions
	\begin{align*}
		a(\boldsymbol{\phi},\lambda;\boldsymbol{\psi},\mu)
		&\leq \|\boldsymbol{\phi}\|_{\varepsilon}\|\boldsymbol{\psi}\|_{\varepsilon},\\
		b(\boldsymbol{\psi}, \mu;\boldsymbol{q})&\leq (\|\curl\boldsymbol{\psi}\|_0+\|\mu\|_0)\|\boldsymbol{q}\|_{H(\div)},
	\end{align*}
	for any $\boldsymbol{\phi}, \boldsymbol{\psi}\in H_0^{1}(\Omega;\mathbb{R}^{3})$,
	$\lambda, \mu\in L_0^2(\Omega)$ and $\boldsymbol{q}\in H_0(\div,\Omega)$.

\begin{lemma}
	For $(\boldsymbol{\psi}, \mu)\in  H_{0}^{1}(\Omega;\mathbb{R}^{3})\times L_0^2(\Omega)$ satisfying $b(\boldsymbol{\psi}, \mu;\boldsymbol{q})=0$ for all~$\boldsymbol{q}\in H_0(\div,\Omega)$, we have
	\begin{align}\label{AEO}
		\|\boldsymbol{\psi}\|_{\varepsilon}^{2}+\|\mu\|_{0}^{2} = a(\boldsymbol{\psi},\mu;\boldsymbol{\psi},\mu).
	\end{align}
\end{lemma}
\begin{proof}
	By taking $\boldsymbol{q}=\curl\,\boldsymbol{\psi}$ in $b(\boldsymbol{\psi}, \mu;\boldsymbol{q})=0$,  we get $\curl\,\boldsymbol{\psi}=0$. Then we have ~$(\mu,\div\boldsymbol{q})=0$ for~$\boldsymbol{q}\in H_0(\div,\Omega)$, so $\mu=0$.
	Therefore,  we get \eqref{AEO}.
\end{proof}

\begin{lemma}
	For $\boldsymbol{q}\in H_{0}(\mathrm{div},\Omega)$, it holds
	\begin{align}\label{AH}
		\|\boldsymbol{q}\|_{H(\div)}\lesssim \sup _{\boldsymbol{\psi}\in H_{0}^{1}(\Omega;\mathbb{R}^{3}),\,\mu\in L_0^2(\Omega)}
		\frac{b(\boldsymbol{\psi}, \mu;\boldsymbol{q})}{\|\boldsymbol{\psi}\|_{\varepsilon}+\|\mu\|_{0}}\,:\,=\beta.
	\end{align}
\end{lemma}
\begin{proof}
	Firstly, utilizing $\div H_{0}^{1}(\Omega;\mathbb{R}^{3})=L_0^2(\Omega)$~\cite[Theorem 1.1]{CostabelMcIntosh2010}, it holds
	\begin{align}\label{divp0}
		\|\div\boldsymbol{q}\|_{0}\lesssim\sup _{\mu\in L_0^2(\Omega)}\frac{(\mu,\div\boldsymbol{q})}{\|\mu\|_{0}}\,\lesssim\beta.
	\end{align}
	On the other hand, for $\boldsymbol{q}\in H_0(\div,\Omega)$, 
	according to the regular decomposition~\cite[Lemma 3.8]{HiptmairJinchao2007}
	\begin{align*}
		H_{0}(\div,\Omega)=\curl(H_{0}^{1}(\Omega;\mathbb{R}^{3}))+ H_{0}^{1}(\Omega;\mathbb{R}^{3}),
	\end{align*}
	there exist $\bar{\boldsymbol{\phi}}\in H_{0}^{1}(\Omega;\mathbb{R}^{3})$ and $\bar{ \boldsymbol{p}}\in H_{0}^{1}(\Omega;\mathbb{R}^{3}) $ such that
	\begin{equation*}
		\boldsymbol{q}=\curl\bar{\boldsymbol{\phi}}+\bar{ \boldsymbol{p}},\qquad \|\bar{\boldsymbol{\phi}}\|_{1}\lesssim\|\curl\bar{\boldsymbol{\phi}}\|_{0}\quad\text{and}\quad\|\bar{\boldsymbol{p}}\|_{1}\lesssim\|\div\boldsymbol{q}\|_{0}.
	\end{equation*}
	Thus,
	\begin{equation*}\|\bar{\boldsymbol{\phi}}\|_{\varepsilon}\leq\varepsilon|\bar{\boldsymbol{\phi}}|_{1}+\|\bar{\boldsymbol{\phi}}\|_{0}+\|\curl\bar{\boldsymbol{\phi}}\|_{0}\lesssim\|\bar{\boldsymbol{\phi}}\|_{1}\lesssim\|\curl\bar{\boldsymbol{\phi}}\|_{0}.
	\end{equation*}
	Then 
	\begin{align*}
		\|\boldsymbol{q}\|_{0}&\leq \|\bar{\boldsymbol{p}}\|_{0}+\|\curl\bar{\boldsymbol{\phi}}\|_{0}
		\lesssim \|\bar{\boldsymbol{p}}\|_{0}+\frac{(\curl\bar{\boldsymbol{\phi}},\boldsymbol{q}-\bar{\boldsymbol{p}})}{\|\bar{\boldsymbol{\phi}}\|_{\varepsilon}} \\
		&\lesssim\|\bar{\boldsymbol{p}}\|_{0}+\sup_{\boldsymbol{\psi}\in H_{0}^{1}(\Omega;\mathbb{R}^{3})}\frac{(\curl\boldsymbol{\psi},\boldsymbol{q}-\bar{\boldsymbol{p}})}{\|\boldsymbol{\psi}\|_{\varepsilon}} \\
		&\lesssim\|\div\boldsymbol{q}\|_{0}+\sup_{\boldsymbol{\psi}\in H_{0}^{1}(\Omega;\mathbb{R}^{3})}\frac{(\curl\boldsymbol{\psi},\boldsymbol{q})}{\|\boldsymbol{\psi}\|_{\varepsilon}}
		\,\leq\|\div\boldsymbol{q}\|_{0}+\beta.
	\end{align*}
	We conclude \eqref{AH} from the above inequality and \eqref{divp0}.
\end{proof}

\begin{theorem}\label{equivalent}
	The decoupled formulation \eqref{decoupleA} is well-posed, and equivalent to
	the primal formulation \eqref{AA0}. That is,
	if $w\in H_{0}^{1}(\Omega)$ is the solution of problem~\eqref{AB}, $(\boldsymbol{\phi}, \boldsymbol{p}, \lambda)\in H_{0}^{1}(\Omega;\mathbb{R}^{3})\times H_{0}(\div,\Omega)\times  L_0^2(\Omega)$ is the solution of problem~\eqref{AC}-\eqref{AD}, and $u\in H_{0}^{1}(\Omega)$ is the solution of problem~\eqref{AE},  then $\lambda=0$, $\div\boldsymbol{p}=0$, $\boldsymbol{\phi}=\nabla u$, and $u\in H_{0}^{2}(\Omega)$ is the solution of the primal formulation \eqref{AA0}.
\end{theorem}
\begin{proof}
	Both \eqref{AB} and \eqref{AE} are the weak formulation of Poisson equation, which are evidently well-posed~\cite[Theorem 1.1.3]{Ciarlet1978}. Thanks to the coercivity \eqref{AEO} and the inf-sup condition~\eqref{AH}, apply the Babu\v{s}a-Brezzi theory~\cite[Theorem 4.2.3]{BoffiBrezziFortin2013} to acquire the well-posedness of the generalized singularly perturbed Stokes-type equation \eqref{AC}-\eqref{AD}. 
	Thus, the decoupled formulation \eqref{decoupleA} is well-posed.
	
	Next we show the equivalence between the decoupled formulation \eqref{decoupleA} and the primal formulation \eqref{AA0}. Take $\boldsymbol{q}= \curl\,\boldsymbol{\phi}$ in (\ref{AD}) to result in $\curl\boldsymbol{\phi}=0$ and $\lambda=0$. 
	Then $\boldsymbol{\phi}=\nabla \tilde{u}\in H_0^{1}(\Omega;\mathbb{R}^{3})$ with some $\tilde{u} \in H_0^{2}(\Omega)$. From~(\ref{AE}), we have $ u=\tilde{u}$ and $\boldsymbol{\phi}=\nabla u$. By setting $\boldsymbol{\psi}=\nabla v$ with $v \in H_0^{2}(\Omega)$ in~(\ref{AC}), we obtain
	\begin{equation*} 
		\varepsilon^2 (\nabla^{2}u,\nabla^{2}v)+(\nabla u,\nabla v)=(\nabla w,\nabla v)\quad\forall\,v \in H_0^{2}(\Omega).
	\end{equation*}
	This, along with (\ref{AB}), leads to (\ref{AA0}).
	
	Finally, choosing $\mu=\div\boldsymbol{p}$ in \eqref{AC0}  derives $\div\boldsymbol{p}=0$.
\end{proof}

Thus, the primal formulation~\eqref{AA0} of the fourth-order elliptic singular perturbation problem~\eqref{eqtwo} can be decoupled into two Poisson equations~\eqref{AB} and~\eqref{AE}, and a generalized singularly perturbed Stokes equation~\eqref{AC}-\eqref{AD} involving the $\curl$ operator. This decoupling not only simplifies the mathematical structure of the fourth-order elliptic singular perturbation problem, but also facilitates the development of efficient finite element methods and the design of fast solvers.
\begin{remark}\rm
	By applying the integration by parts to $(\mu,\div\boldsymbol{p})$ and $(\lambda,\div\boldsymbol{q})$, the formulation \eqref{decoupleA} is equivalent to find $u, w\in H_{0}^{1}(\Omega)$,
	$\boldsymbol{\phi}\in H_{0}^{1}(\Omega;\mathbb{R}^{3})$,
	$\boldsymbol{p}\in L^2(\Omega;\mathbb{R}^{3})$ and $\lambda\in H^1 (\Omega)\cap L_0^2(\Omega)$ such that
	\begin{align*}
		(\nabla w,\nabla v)&=(f,v), \\
		\varepsilon^2(\nabla\boldsymbol{\phi},\nabla\boldsymbol{\psi})+(\boldsymbol{\phi},\boldsymbol{\psi})   +(\curl  \boldsymbol{\psi},\boldsymbol{p})&=(\nabla w,\boldsymbol{\psi}), \\
		(\nabla\mu,\boldsymbol{p})&=0, \\
		(\curl  \boldsymbol{\phi},\boldsymbol{q}) +(\nabla\lambda, \boldsymbol{q}) &=0,\\
		(\nabla u,\nabla \chi)&=(\boldsymbol{\phi},\nabla\chi),
	\end{align*}
	for any $v,\chi\in H_{0}^{1}(\Omega)$,
	$\boldsymbol{\psi}\in H_{0}^{1}(\Omega;\mathbb{R}^{3})$,
	$\boldsymbol{q}\in L^2(\Omega;\mathbb{R}^{3})$ and $\mu\in  H^1 (\Omega)\cap L_0^2(\Omega)$.
\end{remark}

\subsection{Regularity}
When $\varepsilon=0$, the fourth-order elliptic singular perturbation problem~\eqref{eqtwo} reduces to the Poisson equation
\begin{equation}
	\left\{
	\begin{aligned}
		-\Delta u_0&=f\qquad \,\,\mathrm{in}\,\,\Omega,\\
		u_0&=0 \qquad \,\, \mathrm{on}\,\, \partial\Omega.
	\end{aligned}\right.   \label{eqtwo1}
\end{equation}%
Indeed, $u_0=w$, where $w\in H_0^1(\Omega)$ is the solution of problem \eqref{AB}.
\begin{hypothesis}\label{regularityassumption}
	Assume that the Poisson equation \eqref{eqtwo1} satisfies the regularity
	\begin{equation}  \label{H2regularity}
		\|u_0\|_2 = \|w\|_2 \lesssim \|f\|_0,
	\end{equation}
	and that the fourth-order elliptic singular perturbation problem~\eqref{eqtwo} satisfies the regularity
	\begin{equation}  \label{sregularity}
		\|u - u_0\|_1 + \varepsilon \|u\|_2 + \varepsilon^2 \|u\|_3 \lesssim \varepsilon^{1/2} \|f\|_0.
	\end{equation}
\end{hypothesis}
When $\Omega$ is convex, we refer to \cite{MitreaMitreaYan2010,GaoLai2020,Kadlec1964,Talenti1965} for the regularity \eqref{H2regularity}, and~\cite[Lemma
5.1]{NilssenTai2001} and \cite [Lemma 4]{GuzmaneykekhmanNeilan2012} for the regularity \eqref{sregularity} in two and three dimensions. 

Based on Hypothesis \ref{regularityassumption}, we have the following regularity for $\boldsymbol{\phi}$ and $\boldsymbol{p}$.
\begin{lemma}\label{phiregularity}
	Let $(w,\boldsymbol{\phi}=\nabla u, \boldsymbol{p}, 0, u)\in H_{0}^{1}(\Omega)\times H_{0}^{1}(\Omega;\mathbb{R}^{3})\times H_{0}(\div,\Omega)\times  L_0^2(\Omega)\times H_{0}^{2}(\Omega)$ be the solution of the decoupled formulation \eqref{decoupleA}, and let $\boldsymbol{\phi}_0=\nabla u_0\in H_0(\curl,\Omega)$. Under Hypothesis \ref{regularityassumption}, we have
	\begin{align}  \label{sregularitygrad0}
		\|\boldsymbol{\phi}_0\|_1&\lesssim\|f\|_0, \\
		\label{sregularitygrad}
		\|\boldsymbol{\phi}-\boldsymbol{\phi}_0\|_0+\varepsilon\|\boldsymbol{\phi}\|_1  
		+\varepsilon^2\|\boldsymbol{\phi}\|_2
		&\lesssim\varepsilon^{1/2}\|f\|_0, \\
		\label{sregularitycurlp}
		\|\curl\boldsymbol{p}\|_0&\lesssim\varepsilon^{1/2}\|f\|_0.
	\end{align}
\end{lemma}
\begin{proof}
	The regularity results \eqref{sregularitygrad0}-\eqref{sregularitygrad} follow from \eqref{H2regularity}-\eqref{sregularity}, $\boldsymbol{\phi}=\nabla u$ and $\boldsymbol{\phi}_0=\nabla u_0$.
	By \eqref{AC} and $\nabla w=\nabla u_0=\boldsymbol{\phi}_0$,
	\begin{equation}\label{20250603}
		\curl\boldsymbol{p}=\nabla w + \varepsilon^2\Delta\boldsymbol{\phi}-\boldsymbol{\phi}=\varepsilon^2\Delta\boldsymbol{\phi}+\boldsymbol{\phi}_0-\boldsymbol{\phi}.
	\end{equation}
	This together with \eqref{sregularitygrad} implies \eqref{sregularitycurlp}.
\end{proof}

\section{Decoupled nonconforming  finite element method}\label{chap5}
In this section, we propose an optimal decoupled nonconforming finite element method for the fourth-order singular perturbation problem \eqref{eqtwo}, which is robust with respect to the perturbation parameter $\varepsilon$. 

\subsection{Decoupled finite element method}
We propose the following nonconforming finite element method for the decoupled formulation \eqref{decoupleA}: find $w_{h}\in V_{h}^{\grad}$,
$\boldsymbol{\phi}_{h}\in\Phi_{h}$,
$\boldsymbol{p}_h\in V_{h}^{\div}$, $\lambda_h\in\mathcal{Q}_{h}$, and
$u_h\in V_{h}^{\grad}$ such that
\begin{subequations}\label{ImprovedHcurl}
	\begin{align}
		(\nabla w_h,\nabla v)&=(f,v),\label{CA}\\
		\varepsilon^2(\nabla_h\boldsymbol{\phi}_h,\nabla_h\boldsymbol{\psi})+
		(I_h^{\rm ND}\boldsymbol{\phi}_h,I_h^{\rm ND}\boldsymbol{\psi}) + (\curl(I_h^{\rm ND}\boldsymbol{\psi}), \boldsymbol{p}_h)&=(\nabla w_{h},I_h^{\rm ND}\boldsymbol{\psi}),\label{CC} \\
		(\mu,\div\boldsymbol{p}_h)&=0,\label{CC0} \\
		(\curl(I_h^{\rm ND}\boldsymbol{\phi}_h), \boldsymbol{q})-(\lambda_h,\div\boldsymbol{q}) &=0,\label{CD}\\
		(\nabla u_h,\nabla \chi)&=(I_h^{\rm ND}\boldsymbol{\phi}_h,\nabla\chi),\label{CE}
	\end{align}
	for any $v,\,\chi\in V_{h}^{\grad}$, $\boldsymbol{\psi}\in \Phi_{h},\,\mu\in\mathcal{Q}_{h}$, and $\boldsymbol{q}\in  V_{h}^{\div}$.
\end{subequations}
The discrete method \eqref{ImprovedHcurl} is decoupled and can be solved sequentially: first compute $w_h$ from \eqref{CA}; next solve for $\boldsymbol{\phi}_h$, $\boldsymbol{p}_h$, and $\lambda_h$ using \eqref{CC}-\eqref{CD}; and finally obtain $u_h$ from \eqref{CE}.

The interpolation operator $I_h^{\rm ND}$ defined by \eqref{VNDinterp} is incorporated into the decoupled method \eqref{ImprovedHcurl} to ensure optimal convergence rates for arbitrary values of $\varepsilon$, which is inspired by~\cite{WangXuHu2006,WangMeng2007}. 
The interpolation operator introduced in~\cite[Page 116]{WangXuHu2006} and \cite[(2.5) and (3.4)]{WangMeng2007} was designed to address the divergence of the Morley-Wang-Xu method for the Poisson equation, rather than to improve its convergence rate; as a result, the uniform convergence achieved there remains suboptimal at $\mathcal{O}(h^{1/2})$. In contrast, we are the first to employ the interpolation operator $I_h^{\rm ND}$ explicitly to attain optimal convergence for the decoupled method \eqref{ImprovedHcurl}.

For the purpose of the analysis, we introduce the following discrete bilinear forms
\begin{align*}
	a_h(\boldsymbol{\phi}_h,\lambda_h;\boldsymbol{\psi},\mu)&:=\varepsilon^2(\nabla_h\boldsymbol{\phi}_h,\nabla_h\boldsymbol{\psi})+
	(I_h^{\rm ND}\boldsymbol{\phi}_h,I_h^{\rm ND}\boldsymbol{\psi}), \\
	b_h(\boldsymbol{\phi}_h, \lambda_h; \boldsymbol{q})&:=b(I_h^{\rm ND}\boldsymbol{\phi}_h, \lambda_h; \boldsymbol{q})=(\curl(I_h^{\rm ND}\boldsymbol{\phi}_h), \boldsymbol{q})-(\lambda_h,\div\boldsymbol{q}).
\end{align*}
By the commutative diagram \eqref{dcommutesthree}, we have 
\begin{equation}\label{eq:curlcommuteND}
	\curl(I_h^{\rm ND}\boldsymbol{\phi}_h)=\curl\boldsymbol{\phi}_h \quad \textrm{ for } \boldsymbol{\phi}_h\in \Phi_{h},
\end{equation}
which implies
\begin{equation*}
	b_h(\boldsymbol{\phi}_h, \lambda_h; \boldsymbol{q})=(\curl\boldsymbol{\phi}_h, \boldsymbol{q})-(\lambda_h,\div\boldsymbol{q})\quad\forall~\boldsymbol{\phi}_h\in \Phi_{h},\, \boldsymbol{q}\in V_{h}^{\div}.
\end{equation*}

\begin{remark}\rm
	The identity \eqref{eq:curlcommuteND} implies that $\curl\boldsymbol{\phi}_h$ is independent of the face-based DoF \eqref{Phi02}. We use $(\curl(I_h^{\rm ND}\boldsymbol{\phi}_h), \boldsymbol{q})$ rather than $(\curl\boldsymbol{\phi}_h, \boldsymbol{q})$ in the discrete method \eqref{ImprovedHcurl} to reflect this property. Moreover, $I_h^{\rm ND}\boldsymbol{\phi}_h$ belongs to $H_0(\curl,\Omega)$, unlike $\boldsymbol{\phi}_h$, and this distinction is essential for achieving the optimal and uniform convergence of the discrete method \eqref{ImprovedHcurl}.
\end{remark}

We define the following discrete norm for function $\boldsymbol{\psi}\in H_0^{2}(\Omega;\mathbb R^3)+\Phi_{h}$:  
\begin{equation*}
	\|\boldsymbol{\psi}\|_{\varepsilon,h}:=(\varepsilon^2 |\boldsymbol{\psi}|_{1,h}^2+\|\curl(I_h^{\rm ND}\boldsymbol{\psi})\|_0^2+\|I_h^{\rm ND}\boldsymbol{\psi}\|_0^2)^{1/2}.
\end{equation*}
By the identity \eqref{eq:curlcommuteND},
\begin{equation}\label{eq:normequivepscurl}
	\|\boldsymbol{\psi}\|_{\varepsilon,h}=(\varepsilon^2 |\boldsymbol{\psi}|_{1,h}^2+\|\curl\boldsymbol{\psi}\|_0^2+\|I_h^{\rm ND}\boldsymbol{\psi}\|_0^2)^{1/2}\qquad\forall~\boldsymbol{\psi}\in \Phi_{h}.
\end{equation}
Clearly, we have the following discrete continuity
	\begin{align*}
		a_h(\boldsymbol{\phi},\lambda;\boldsymbol{\psi},\mu)
		&\leq \|\boldsymbol{\phi}\|_{\varepsilon,h}\|\boldsymbol{\psi}\|_{\varepsilon,h},\\
		b_h(\boldsymbol{\psi}, \mu;\boldsymbol{q})&\leq (\|\curl(I_h^{\rm ND}\boldsymbol{\psi})\|_0+\|\mu\|_0)\|\boldsymbol{q}\|_{H(\div)},
	\end{align*}
	for any $\boldsymbol{\phi},\,\boldsymbol{\psi}\in \Phi_{h},\, \lambda,\,\mu\in\mathcal{Q}_{h}$, and $\boldsymbol{q}\in  V_{h}^{\div}$.


\begin{lemma}
	For $(\boldsymbol{\psi}, \mu)\in\Phi_{h}\times\mathcal{Q}_{h}$ satisfying $b_h(\boldsymbol{\psi}, \mu;\boldsymbol{q})=0$ for all $\boldsymbol{q}\in  V_{h}^{\div}$, then
	\begin{align}\label{BEphi}
		\|\boldsymbol{\psi}\|_{\varepsilon,h}^{2}+\|\mu\|_{0}^{2}= \varepsilon^2|\boldsymbol{\psi}|_{1,h}^2+\|I_h^{\rm ND}\boldsymbol{\psi}\|_0^2= a_h(\boldsymbol{\psi},\mu;\boldsymbol{\psi},\mu).
	\end{align}
\end{lemma}
\begin{proof}
	By taking $\boldsymbol{q}=\curl(I_h^{\rm ND}\boldsymbol{\psi})$ in $b_h(\boldsymbol{\psi}, \mu;\boldsymbol{q})=0$,  we get $\curl(I_h^{\rm ND}\boldsymbol{\psi})=0$. Then we have  $(\mu,\div\boldsymbol{q})=0$, which together with $\div V_{h}^{\div}=\mathcal{Q}_{h}$ implies $\mu=0$.
	Therefore,~\eqref{BEphi} holds.
\end{proof}

With the help of the commutative diagram~\eqref{dcommutes2} and~\eqref{dcommutesthree}, we can establish the $L^2$-stable decomposition of space $ V_{h}^{\div}$.
\begin{lemma}\label{lem:L2stabledecompRT}
	For $\boldsymbol{q}\in  V_{h}^{\div}$, these exist $\boldsymbol{p}_h\in V_{h}^{\div}$ and $\boldsymbol{\psi}_h \in \Phi_{h}$, such that 
	\begin{equation*}
		\boldsymbol{q}=\boldsymbol{p}_h+\curl\boldsymbol{\psi}_h, \quad \|\boldsymbol{p}_h\|_{1,h}\lesssim\|\div\boldsymbol{q}\|_{0}\quad\text{and}\quad\|\boldsymbol{\psi}_h\|_{1,h}\lesssim \|\curl\boldsymbol{\psi}_h\|_{0}.
	\end{equation*}
\end{lemma}
\begin{proof}
	Since $\boldsymbol{q}\in V_{h}^{\div}\subset H_0(\div,\Omega)$, by $\div H_0^1(\Omega;\mathbb R^3)=L_0^2(\Omega)$, there exists a $\boldsymbol{p}\in H_0^1(\Omega;\mathbb R^3)$ such that 
	\begin{equation*}
		\div\boldsymbol{p}=\div\boldsymbol{q},\quad \|\boldsymbol{p}\|_{1}\lesssim\|\div\boldsymbol{q}\|_{0}.
	\end{equation*}
	Let $\boldsymbol{p}_h=I_{h}^{\rm{RT}}\boldsymbol{p}\in V_{h}^{\div}$. Then by the commutative diagram~\eqref{dcommutes2} and \eqref{eq:Ihdivprop2}, we have 
	\begin{equation*}
		\div\boldsymbol{p}_h=Q_h(\div\boldsymbol{p})=\div\boldsymbol{p},\quad \|\boldsymbol{p}_h\|_{1,h}\lesssim\|\boldsymbol{p}\|_{1}\lesssim\|\div\boldsymbol{q}\|_{0}.
	\end{equation*}
	By $\curl H_0^1(\Omega;\mathbb R^3)=H_0(\div,\Omega)\cap\ker(\div)$~\cite[Theorem 1.1]{CostabelMcIntosh2010} and $\div(\boldsymbol{q}-\boldsymbol{p}_h)=0$, there exists a $\boldsymbol{\psi}\in H_0^1(\Omega;\mathbb R^3)$ such that 
	\begin{equation*}
		\curl\boldsymbol{\psi}=\boldsymbol{q}-\boldsymbol{p}_h\in V_{h}^{\div},\quad \|\boldsymbol{\psi}\|_{1}\lesssim\|\curl\boldsymbol{\psi}\|_{0}.
	\end{equation*}
	Let $\boldsymbol{\psi}_h=I_h^{\Phi}\boldsymbol{\psi}\in \Phi_{h}$. We obtain from \eqref{interpolation02} and \eqref{BEqs} that
	\begin{equation*}
		\curl\boldsymbol{\psi}_h=\curl(I_h^{\Phi}\boldsymbol{\psi})=I^{\rm RT}_h(\curl\boldsymbol{\psi})=\curl\boldsymbol{\psi},\quad \|\boldsymbol{\psi}_h\|_{1,h}\lesssim \|\boldsymbol{\psi}\|_{1}\lesssim\|\curl\boldsymbol{\psi}\|_{0}.
	\end{equation*}
	This ends the proof.
\end{proof}
\begin{lemma}
	It holds the discrete inf-sup condition
	\begin{align}\label{BEp}
		\|\boldsymbol{q}\|_{H(\div)}\lesssim \sup _{\boldsymbol{\psi}\,\in \Phi_{h}, \mu\,\in \mathcal{Q}_{h}}
		\frac{b_h(\boldsymbol{\psi}, \mu;\boldsymbol{q})}{\|\boldsymbol{\psi}\|_{\varepsilon,h}+\|\mu\|_{0}}\,:=\,\beta_{h}\qquad\forall~\boldsymbol{q}\in   V_{h}^{\div}.
	\end{align}
\end{lemma}
\begin{proof}
	For $\boldsymbol{q}\in  V_{h}^{\div}$, by $\div V_{h}^{\div}=\mathcal{Q}_{h}$, we have
	\begin{equation}\label{divph}
		\|\div\boldsymbol{q}\|_{0}\lesssim\sup _{\mu\in \mathcal{Q}_{h}}
		\frac{(\mu,\div\boldsymbol{q})}{\|\mu\|_{0}}\,\lesssim\beta_{h}.
	\end{equation}
	Applying Lemma~\ref{lem:L2stabledecompRT} to $\boldsymbol{q}$, it follows that
	\begin{equation*}
		\|\boldsymbol{\psi}_h\|_{\varepsilon,h}\lesssim \varepsilon|\boldsymbol{\psi}_h|_{1,h}+\|\boldsymbol{\psi}_h\|_{1,h}\lesssim\|\curl\boldsymbol{\psi}_h\|_{0}.
	\end{equation*}
	Then by \eqref{eq:normequivepscurl}, 
	we deduce
	\begin{align*}
		\|\boldsymbol{q}\|_{0}
		&\lesssim\|\curl\boldsymbol{\psi}_h\|_{0}+\|\boldsymbol{p}_h\|_{0}
		=\frac{(\curl\boldsymbol{\psi}_h,\boldsymbol{q}-\boldsymbol{p}_h)}
		{\|\curl\boldsymbol{\psi}_h\|_{0}}+\|\boldsymbol{p}_h\|_{0}\\
		&\lesssim\sup_{\boldsymbol{\psi}\in \Phi_{h}}\frac{(\curl\boldsymbol{\psi},\boldsymbol{q}-\boldsymbol{p}_h)}
		{\|\boldsymbol{\psi}\|_{\varepsilon,h}} +\|\boldsymbol{p}_h\|_{0} \lesssim
		\sup_{\boldsymbol{\psi}\in \Phi_{h}}\frac{(\curl\boldsymbol{\psi},\boldsymbol{q})}
		{\|\boldsymbol{\psi}\|_{\varepsilon,h}} +\|\boldsymbol{p}_h\|_{0} \\
		&\lesssim\beta_{h} +\|\div\boldsymbol{q}\|_{0}. 
	\end{align*}
	Combining the above inequality and~\eqref{divph}  yields the desired result~\eqref{BEp}. 
\end{proof}

Thanks to the discrete coercivity~\eqref{BEphi} and the discrete inf-sup condition \eqref{BEp}, we derive the following discrete stability by using the Babu\v{s}a-Brezzi theory~\cite[Theorem 4.2.3]{BoffiBrezziFortin2013}.  
\begin{lemma}\label{stabilityresults}
	We have the discrete stability
	\begin{align}\label{BG}
		&\|\tilde{\boldsymbol{\phi}}_{h}\|_{\varepsilon,h}+\|\tilde{\lambda}_{h}\|_{0}+\|\tilde{\boldsymbol{p}}_{h}\|_{H(\div)}\notag\\
		\lesssim &\sup _{\boldsymbol{\psi}\,\in \Phi_{h},\,
			\mu\,\in \mathcal{Q}_{h},\,\boldsymbol{q}\in V_{h}^{\div}}
		\frac{a_h(\tilde{\boldsymbol{\phi}}_{h},\tilde{\lambda}_{h};\boldsymbol{\psi},\mu)+b_h(\boldsymbol{\psi}, \mu;\tilde{\boldsymbol{p}}_{h})+b_h(\tilde{\boldsymbol{\phi}}_{h}, \tilde{\lambda}_{h};\boldsymbol{q})}
		{\|\boldsymbol{\psi}\|_{\varepsilon,h}+\|\mu\|_{0}+\|\boldsymbol{q}\|_{H(\div)}}
	\end{align}
	for any $\tilde{\boldsymbol{\phi}}_{h}\in\Phi_{h}$, $\tilde{\lambda}_{h}\in \mathcal{Q}_{h}$, and $\tilde{\boldsymbol{p}}_{h}\in V_{h}^{\div}$. It follows that the  mixed finite element method~\eqref{CC}-\eqref{CD} is well-posed.
\end{lemma}

\begin{theorem}
	The decoupled finite element method \eqref{ImprovedHcurl} is well-posed. We have 
	$\lambda_h=0$, $\div\boldsymbol{p}_h=0$, $I_h^{\rm ND}\boldsymbol{\phi}_{h}=\nabla u_h$, $\boldsymbol{\phi}_{h}=\nabla u_h^W$, and $u_h=I_h^{\grad}u_h^W$, where $u_{h}^W\in W_{h}$ is the solution of the following nonconforming finite element method:
	\begin{equation}\label{fourthpertbncfminterp}
		\varepsilon^2 (\nabla_h^2u_h^W, \nabla_h^2v)+(\nabla(I_h^{\grad}u_h^W),\nabla(I_h^{\grad}v))=(f,I_h^{\grad}v)\quad\forall~v\in W_{h}.
	\end{equation}
\end{theorem}
\begin{proof}
	Both \eqref{CA} and \eqref{CE} are Lagrange finite element methods for the Poisson equation and are therefore well-posed. Together with Lemma~\ref{stabilityresults}, this implies the well-posedness of the decoupled finite element method \eqref{ImprovedHcurl}.
	
	By choosing $\boldsymbol{q} = \curl(I_h^{\rm ND}\boldsymbol{\phi}_{h})$ in~\eqref{CD}, we obtain $\curl (I_h^{\rm ND} \boldsymbol{\phi}_{h})= 0$ and $\lambda_h = 0$. Then it follows from the finite element de Rham complex \eqref{discretecomplex3} and \eqref{CE} that $I_h^{\rm ND}\boldsymbol{\phi}_{h}=\nabla u_h$. Set $\mu=\div\boldsymbol{p}_h$ in \eqref{CC0} to acquire $\div\boldsymbol{p}_h=0$. By \eqref{eq:curlcommuteND}, $\curl\boldsymbol{\phi}_{h}=\curl(I_h^{\rm ND}\boldsymbol{\phi}_{h})=0$, which together with the complex \eqref{discretecomplex2} indicates $\boldsymbol{\phi}_{h}=\nabla u_h^W$ for some $u_h^W\in W_h$.
	By the commutative diagram \eqref{dcommutesthree}, it follows that
	\begin{equation*}
		\nabla u_h=I_h^{\rm ND}\boldsymbol{\phi}_{h}=I_h^{\rm ND}(\nabla u_h^W)=\nabla(I_h^{\grad}u_h^W).
	\end{equation*}
	So $u_h=I_h^{\grad}u_h^W$.
	
	By taking $\boldsymbol{\psi}=\nabla v$ with $v\in W_h$, the equation \eqref{CC} becomes
	\begin{equation*}
		\varepsilon^2 (\nabla_h^2u_h^W, \nabla_h^2v) + (\nabla(I_h^{\grad}u_h^W), \nabla(I_h^{\grad}v)) =(\nabla w_{h}, \nabla(I_h^{\grad}v))\quad\forall\,v\in W_h.
	\end{equation*}
	This combined with \eqref{CA} implies that $u_h^W$ satisfies \eqref{fourthpertbncfminterp}.
\end{proof}

\begin{remark}\rm
	When implementing the terms in the decoupled method \eqref{ImprovedHcurl} that involve the interpolation operator $I_h^{\rm ND}$, the relation $I_h^{\rm ND}\Phi_{h}=V_{h}^{\rm ND}$ allows us to directly use the basis functions of the space $V_{h}^{\rm ND}$ for assembly. Consequently, there is no need to compute the interpolation matrix of $I_h^{\rm ND}$, and its use does not increase the computational cost.
\end{remark}

\subsection{Robust error analysis}
Let $Q^{\rm RT}_h\boldsymbol{p}\in \curl\Phi_{h}=V_{h}^{\div}\cap\ker(\div)$ be the $L^2$-projection of $\boldsymbol{p}$ onto space $\curl\Phi_{h}$. By the commutative diagram \eqref{dcommutesdeRham}, it follows that $\div(I^{\rm RT}_h\boldsymbol{p})=Q_h(\div\boldsymbol{p})=0$, then $I^{\rm RT}_h\boldsymbol{p}\in V_{h}^{\div}\cap\ker(\div)$. Thus, for $\boldsymbol{p}\in H_0(\div,\Omega)\cap\ker(\div)$, we have
\begin{equation}\label{eq:202506041}
	\|\boldsymbol{p}-Q^{\rm RT}_h\boldsymbol{p}\|_0 \leq \|\boldsymbol{p}-I^{\rm RT}_h\boldsymbol{p}\|_0.
\end{equation}
The projection $Q^{\rm RT}_h\boldsymbol{p}$ is vital in later analysis.

For the Lagrange element method~\eqref{CA}, applying the standard error analysis in~\cite[Theorem 3.2.2]{Ciarlet1978}, we have the following estimate for $|w-w_h|_1$.
\begin{lemma}
	Let $w=u_0\in H_0^1(\Omega)$ be the solution of the Poisson equation \eqref{eqtwo1}, $w_{h}\in V_{h}^{\grad}$ be the solution of the finite element method \eqref{CA}. Assume $w=u_0\in H^s(\Omega)$ with $2\leq s\leq 3$. Then it holds
	\begin{equation}\label{estimatewh}
		|w-w_h|_1\lesssim h^{s-1}|u_0|_s.
	\end{equation}
\end{lemma}

Next, we provide a consistent error estimate for the  mixed method \eqref{CC}-\eqref{CD}.
\begin{lemma}
	Let $(\boldsymbol{\phi}, \boldsymbol{p}, 0)\in H_{0}^{1}(\Omega;\mathbb{R}^{3})\times H_{0}(\div,\Omega)\times  L_0^2(\Omega)$ be the solution of the generalized singularly perturbed Stokes-type equation \eqref{AC}-\eqref{AD}. Assume $u_0\in H^s(\Omega)$ with $2\leq s\leq 3$. Let $0\leq r\leq 1$. Under Hypothesis \ref{regularityassumption}, we have
	\begin{align}\label{IPhiestimatef}
		\varepsilon|\boldsymbol{\phi}-I_h^{\Phi}\boldsymbol{\phi}|_{1,h}&\lesssim \varepsilon^{r-1/2} h^{1-r} \|f\|_{0}, \\
		\label{PiNDestimatef}
		\|\boldsymbol{\phi}-\Pi_h^{\rm ND}\boldsymbol{\phi}\|_0&\lesssim \varepsilon^{r-1/2} h^{1-r} \|f\|_{0} + h^{s-1}|u_0|_{s}.
	\end{align}
\end{lemma}
\begin{proof}
	Using \eqref{BEqs} and the regularities \eqref{sregularitygrad0}-\eqref{sregularitygrad}, we have
	\begin{align*}
		\varepsilon|\boldsymbol{\phi}-I_h^{\Phi}\boldsymbol{\phi}|_{1,h}&\lesssim \varepsilon h^{1-r}|\boldsymbol{\phi}|_1^r|\boldsymbol{\phi}|_2^{1-r}\lesssim \varepsilon^{r-1/2} h^{1-r} \|f\|_{0}.
	\end{align*}
	Hence, \eqref{IPhiestimatef} is true.
	It follows from \eqref{eq:IhND2} and the regularities \eqref{sregularitygrad0}-\eqref{sregularitygrad} that
	\begin{align*}
		\|\boldsymbol{\phi}-\Pi_h^{\rm ND}\boldsymbol{\phi}\|_0&\leq \|(\boldsymbol{\phi}-\boldsymbol{\phi}_0)-\Pi_h^{\rm ND}(\boldsymbol{\phi}-\boldsymbol{\phi}_0)\|_0+ \|\boldsymbol{\phi}_0-\Pi_h^{\rm ND}\boldsymbol{\phi}_0\|_0 \\
		&\lesssim h^{1-r}\|\boldsymbol{\phi}-\boldsymbol{\phi}_0\|_0^r|\boldsymbol{\phi}-\boldsymbol{\phi}_0|_1^{1-r}+ \|\boldsymbol{\phi}_0-\Pi_h^{\rm ND}\boldsymbol{\phi}_0\|_0 \\
		&\lesssim \varepsilon^{r-1/2} h^{1-r} \|f\|_{0} + h^{s-1}|u_0|_{s}.
	\end{align*}
	This implies \eqref{PiNDestimatef}.
\end{proof}

\begin{lemma}
	Assume Hypothesis \ref{regularityassumption} holds and $u_0\in H^s(\Omega)$ with $2\leq s\leq 3$.
	For any  $(\boldsymbol{\psi}, \mu)\in\Phi_{h}\times\mathcal{Q}_{h}$, we have the following consistent error estimate 
	\begin{equation}\label{consistenterror01}
		\begin{aligned}
			&a_h(I_h^{\Phi}\boldsymbol{\phi},0;\boldsymbol{\psi},\mu)+b_h(\boldsymbol{\psi}, \mu; Q^{\rm RT}_h\boldsymbol{p})-(\nabla w_{h},I_h^{\rm ND}\boldsymbol{\psi})\\
			&\qquad\qquad\qquad\qquad\lesssim (\varepsilon^{r-1/2} h^{1-r} \|f\|_{0}+h^{s-1}|u_0|_{s})\|\boldsymbol{\psi}\|_{\varepsilon,h},\quad 0\leq r\leq 1.
		\end{aligned}
	\end{equation}
\end{lemma}
\begin{proof}
	Using \eqref{INDIPhiPiND}, $\curl\boldsymbol{p}=\nabla w+ \varepsilon^2\Delta\boldsymbol{\phi}-\boldsymbol{\phi}$ by \eqref{20250603}, and integration by parts, we have
	\begin{equation}\label{202506031}
		\begin{aligned}
			&\quad\; a_h(I_h^{\Phi}\boldsymbol{\phi},0;\boldsymbol{\psi},\mu)+b_h(\boldsymbol{\psi}, \mu; Q^{\rm RT}_h\boldsymbol{p})-(\nabla w_h,I_h^{\rm ND}\boldsymbol{\psi}) \\
			&= \varepsilon^2(\nabla_h(I_h^{\Phi}\boldsymbol{\phi}),\nabla_{h}\boldsymbol{\psi})+(\Pi_h^{\rm ND}\boldsymbol{\phi},I_h^{\rm ND}\boldsymbol{\psi})+(\curl(I_h^{\rm ND}\boldsymbol{\psi}), \boldsymbol{p})
			-(\nabla w_h,I_h^{\rm ND}\boldsymbol{\psi}) \\
			&= \varepsilon^2(\nabla_h(I_h^{\Phi}\boldsymbol{\phi}),\nabla_{h}\boldsymbol{\psi}) + \varepsilon^2(\Delta\boldsymbol{\phi},I_h^{\rm ND}\boldsymbol{\psi}) \\
			&\quad\; +(\Pi_h^{\rm ND}\boldsymbol{\phi}-\boldsymbol{\phi}+\nabla(w-w_h),I_h^{\rm ND}\boldsymbol{\psi}) \\
			&=: I_1+I_2,
		\end{aligned}
	\end{equation}
	where
	\begin{align*}
		I_1&=\varepsilon^2(\nabla_h(I_h^{\Phi}\boldsymbol{\phi}),\nabla_{h}\boldsymbol{\psi}) + \varepsilon^2(\Delta\boldsymbol{\phi},I_h^{\rm ND}\boldsymbol{\psi}), \\
		I_2&=(\Pi_h^{\rm ND}\boldsymbol{\phi}-\boldsymbol{\phi}+\nabla(w-w_h),I_h^{\rm ND}\boldsymbol{\psi}).
	\end{align*}
	Next, we estimate $I_1$ and $I_2$ separately.
	
	By \eqref{BEqs} and the regularity \eqref{sregularitygrad}, one has
	\begin{equation*}
		I_1\lesssim \varepsilon^2|\boldsymbol{\phi}|_1|\boldsymbol{\psi}|_{1,h}+\varepsilon^2|\boldsymbol{\phi}|_2\|I_h^{\rm ND}\boldsymbol{\psi}\|_{0}\lesssim \varepsilon^{1/2}\|f\|_0 \|\boldsymbol{\psi}\|_{\varepsilon,h}.
	\end{equation*}
	On the other hand,
	by applying the integration by parts, the weak continuity \eqref{Phiweakcontinuity1}, the estimate \eqref{BEqs} of $I_h^{\Phi}$, the estimate \eqref{eq:IhNDestimes} of $I_h^{\rm ND}$ and the regularity \eqref{sregularitygrad}, we derive
	\begin{align*}
		I_1&=\varepsilon^2(\nabla_h(I_h^{\Phi}\boldsymbol{\phi}-\boldsymbol{\phi}),\nabla_{h}\boldsymbol{\psi}) + \varepsilon^2(\Delta\boldsymbol{\phi}, I_h^{\rm ND}\boldsymbol{\psi}-\boldsymbol{\psi}) + \varepsilon^2\big((\nabla\boldsymbol{\phi},\nabla_{h}\boldsymbol{\psi}) + (\Delta\boldsymbol{\phi},\boldsymbol{\psi})\big) \\
		&=\varepsilon^2(\nabla_h(I_h^{\Phi}\boldsymbol{\phi}-\boldsymbol{\phi}),\nabla_{h}\boldsymbol{\psi}) + \varepsilon^2(\Delta\boldsymbol{\phi},I_h^{\rm ND}\boldsymbol{\psi}-\boldsymbol{\psi}) \\
		&\quad + \sum_{T\in\mathcal{T}_h}\sum_{F\in\Delta_{2}(T)}\varepsilon^2(\partial_{n}\boldsymbol{\phi}-Q_F^0(\partial_{n}\boldsymbol{\phi}), \boldsymbol{\psi}-Q_F^0\boldsymbol{\psi})_{F} \\
		&\lesssim \varepsilon^2h|\boldsymbol{\phi}|_2|\boldsymbol{\psi}|_{1,h} \lesssim \varepsilon^{-1/2}h\|f\|_0 \|\boldsymbol{\psi}\|_{\varepsilon,h}.
	\end{align*}
	Combine the above two estimates to obtain
	\begin{equation}\label{202506032}
		I_1\lesssim \varepsilon^{r-1/2} h^{1-r} \|f\|_{0}\|\boldsymbol{\psi}\|_{\varepsilon,h}.
	\end{equation}
	It follows from estimates \eqref{PiNDestimatef} and \eqref{estimatewh} that
	\begin{equation}\label{202506033}
		\begin{aligned}
			I_2&\leq\|\Pi_h^{\rm ND}\boldsymbol{\phi}-\boldsymbol{\phi}+\nabla(w-w_h)\|_0\|I_h^{\rm ND}\boldsymbol{\psi}\|_0 \\
			&\lesssim (\varepsilon^{r-1/2} h^{1-r} \|f\|_{0} + h^{s-1}|u_0|_{s})\|\boldsymbol{\psi}\|_{\varepsilon,h}.
		\end{aligned}
	\end{equation}
	Therefore, \eqref{consistenterror01} follows from \eqref{202506031}-\eqref{202506033}.
\end{proof}

\begin{theorem}\label{thm:Bncerrorestima1}
	Let $(w, \boldsymbol{\phi}, \boldsymbol{p}, 0, u)$ be the solution of the decoupled variational formulation~\eqref{decoupleA}, $u_0\in H_0^1(\Omega)$ be the solution of the Poisson equation~\eqref{eqtwo1}, $\boldsymbol{\phi}_0=\nabla u_0\in H_0(\curl,\Omega)$,
	and $(w_h, \boldsymbol{\phi}_h, \boldsymbol{p}_h, 0, u_h)$ be the solution of the decoupled variational formulation~\eqref{ImprovedHcurl}. Assume Hypothesis \ref{regularityassumption} holds and $u_0\in H^s(\Omega)$ with $2\leq s\leq 3$.
	For $0\leq r\leq 1$, we have that
	\begin{align}
		\label{varepsilonIh2}
		\varepsilon|\boldsymbol{\phi}-\boldsymbol{\phi}_{h}|_{1,h}+\|\boldsymbol{\phi}-I_h^{\rm ND} \boldsymbol{\phi}_h\|_0 & \lesssim  \varepsilon^{r-1/2} h^{1-r}\|f\|_{0}+ h^{s-1}|u_0|_{s}, \\
		\label{varepsilonIh3}
		\varepsilon|\boldsymbol{\phi}_0-\boldsymbol{\phi}_{h}|_{1,h}+\|\boldsymbol{\phi}_0-I_h^{\rm ND} \boldsymbol{\phi}_h\|_0 & \lesssim  \varepsilon^{1/2}\|f\|_{0}+ h^{s-1}|u_0|_{s}, \\
		\label{eq:errorestimatauhH03}
		|u-u_h|_1 & \lesssim  \varepsilon^{r-1/2} h^{1-r}\|f\|_{0}+ h^{s-1}|u_0|_{s}, \\
		\label{eq:errorestimatauhH04}
		|u_0-u_h|_1 & \lesssim  \varepsilon^{1/2}\|f\|_{0}+ h^{s-1}|u_0|_{s}.
	\end{align}
	Furthermore, when $\boldsymbol{p}\in H^1(\Omega;\mathbb R^3)$ satisfying $\|\boldsymbol{p}\|_1\lesssim \|\curl \boldsymbol{p}\|_0$, we have
	\begin{equation}\label{varepsilonIhp}
		\|\boldsymbol{p}-\boldsymbol{p}_h\|_0\lesssim \varepsilon^{r-1/2} h ^{1-r}\|f\|_{0}+ h^{s-1}|u_0|_{s}.
	\end{equation}
\end{theorem}
\begin{proof}
	By \eqref{CC}-\eqref{CC0} and \eqref{consistenterror01}, we have
	\begin{align*}
		&\quad\; a_h(I_h^{\Phi}\boldsymbol{\phi}-\boldsymbol{\phi}_h,0;\boldsymbol{\psi},\mu)+b_h(\boldsymbol{\psi}, \mu; Q^{\rm RT}_h\boldsymbol{p}-\boldsymbol{p}_h) \\
		&=a_h(I_h^{\Phi}\boldsymbol{\phi},0;\boldsymbol{\psi},\mu)+b_h(\boldsymbol{\psi}, \mu; Q^{\rm RT}_h\boldsymbol{p})-(\nabla w_{h},I_h^{\rm ND}\boldsymbol{\psi}) \\
		& \lesssim (\varepsilon^{r-1/2} h^{1-r} \|f\|_{0} +h^{s-1}|u_0|_{s})\|\boldsymbol{\psi}\|_{\varepsilon,h}.
	\end{align*}
	Using \eqref{CD}, \eqref{INDIPhiPiND}, the commutative diagram \eqref{dcommutesdeRham} and the fact $\boldsymbol{\phi}=\nabla u$, it holds that
	\begin{equation*}
		b_h(I_h^{\Phi}\boldsymbol{\phi}-\boldsymbol{\phi}_h,0; \boldsymbol{q})=b_h(I_h^{\Phi}\boldsymbol{\phi},0; \boldsymbol{q})=(\curl(\Pi_h^{\rm ND}\boldsymbol{\phi}), \boldsymbol{q})=(I_h^{\rm RT}(\curl\boldsymbol{\phi}), \boldsymbol{q})=0.
	\end{equation*}
	Then employing the discrete stability \eqref{BG} with $\tilde{\boldsymbol{\phi}}_{h}=I_h^{\Phi}\boldsymbol{\phi}-\boldsymbol{\phi}_h$, $\tilde{\lambda}_{h}=0$ and $\tilde{\boldsymbol{p}}_{h}=Q^{\rm RT}_h\boldsymbol{p}-\boldsymbol{p}_h$, we get 
	\begin{equation}\label{eq:202506042}
		\|I_h^{\Phi}\boldsymbol{\phi}-\boldsymbol{\phi}_h\|_{\varepsilon,h}+\|Q^{\rm RT}_h\boldsymbol{p}-\boldsymbol{p}_h\|_{0}
		\lesssim \varepsilon^{r-1/2} h^{1-r} \|f\|_{0}+h^{s-1}|u_0|_{s}.
	\end{equation}
	This together with 
	\eqref{IPhiestimatef}-\eqref{PiNDestimatef} and \eqref{INDIPhiPiND} implies \eqref{varepsilonIh2}.
	The estimate \eqref{varepsilonIh3} follows from \eqref{varepsilonIh2} and the regularity \eqref{sregularitygrad0}-\eqref{sregularitygrad}.
	
	Since $\boldsymbol{\phi} = \nabla u$, $\boldsymbol{\phi}_0 = \nabla u_0$, and $I_h^{\rm ND} \boldsymbol{\phi}_h = \nabla u_h$, the estimates \eqref{eq:errorestimatauhH03}-\eqref{eq:errorestimatauhH04} follow directly from \eqref{varepsilonIh2}-\eqref{varepsilonIh3}.
	
	Finally, the estimate \eqref{varepsilonIhp} holds from \eqref{eq:202506042}, \eqref{eq:202506041}, \eqref{eq:Ihdivprop2} and \eqref{sregularitycurlp}.
\end{proof}

Next, we apply a duality argument to estimate the errors $\|\boldsymbol{\phi} - I_h^{\rm ND}\boldsymbol{\phi}_{h}\|_0$ and $\|u-u_{h}\|_{0}$. 
Let $\hat{u}\in H_0^2(\Omega)$ satisfy the following dual problem 
\begin{equation}\label{bdual1}
\varepsilon^2\Delta^2\hat{u} - \Delta\hat{u}=-\div(\boldsymbol{\phi} - I_h^{\rm ND}\boldsymbol{\phi}_{h})\in H^{-1}(\Omega).
\end{equation}%
It can be written as the form of the biharmonic equation
\begin{equation}\label{biharmonicdual1}
\Delta^2\hat{u} = \varepsilon^{-2}(\Delta\hat{u} - \div(\boldsymbol{\phi} - I_h^{\rm ND}\boldsymbol{\phi}_{h})).
\end{equation}%
Assume problem \eqref{biharmonicdual1} has the following $H^3$ regularity
\begin{equation}\label{iharmonicdualregularity}
\varepsilon^2\|\hat{u}\|_3\lesssim	|\hat{u}|_1 + \|\boldsymbol{\phi} - I_h^{\rm ND}\boldsymbol{\phi}_{h}\|_0.
\end{equation}
We refer to \cite{MazyaRossmann2010,Grisvard1985,Grisvard1992,BlumRannacher1980} for regularity results of the biharmonic equation on polytopal domains.

\begin{lemma}
Assume the $H^3$ regularity \eqref{iharmonicdualregularity} holds.
The dual problem \eqref{bdual1} has the regularity
\begin{equation}\label{uhatregularity}
\|\hat{u}\|_1+ \varepsilon\|\hat{u}\|_2+ \varepsilon^2\|\hat{u}\|_3\lesssim	 \|\boldsymbol{\phi} - I_h^{\rm ND}\boldsymbol{\phi}_{h}\|_0.
\end{equation}	
\end{lemma}
\begin{proof}
Multiplying \eqref{bdual1} by $\hat{u}$ and integrating by parts, we have
\begin{equation*}
\varepsilon^2|\hat{u}|_2^2 + |\hat{u}|_1^2 = (\boldsymbol{\phi} - I_h^{\rm ND}\boldsymbol{\phi}_{h}, \nabla \hat{u}).
\end{equation*}	
This means that
\begin{equation*}
\varepsilon^2|\hat{u}|_2^2 + \frac{1}{2}|\hat{u}|_1^2 \leq \frac{1}{2}\|\boldsymbol{\phi} - I_h^{\rm ND}\boldsymbol{\phi}_{h}\|_0^2.
\end{equation*}
Thus, the regularity \eqref{uhatregularity} follows from \eqref{iharmonicdualregularity} and the last inequality.
\end{proof}

\begin{lemma}
Assume the $H^3$ regularity \eqref{iharmonicdualregularity} holds.
There exist $\hat{\boldsymbol{\phi}}\in H_{0}^{1}(\Omega;\mathbb{R}^{3})\cap H^{2}(\Omega;\mathbb{R}^{3})$ and $\hat{\boldsymbol{p}}\in H^1(\Omega; \mathbb{R}^{3})$ such that
\begin{equation}\label{bdual2}
-\varepsilon^2\Delta\hat{\boldsymbol{\phi}}+\hat{\boldsymbol{\phi}}+\curl \hat{\boldsymbol{p}} =\boldsymbol{\phi} - I_h^{\rm ND}\boldsymbol{\phi}_{h}, \;\;\curl\hat{\boldsymbol{\phi}}=0 \quad \textrm{ in } \Omega,
\end{equation}%
\begin{align}\label{bdual2regularity}
\|\hat{\boldsymbol{\phi}}\|_0 +\varepsilon\|\hat{\boldsymbol{\phi}}\|_1+	\varepsilon^2\|\hat{\boldsymbol{\phi}}\|_2 + \|\hat{\boldsymbol{p}}\|_1 \lesssim \|\boldsymbol{\phi} - I_h^{\rm ND}\boldsymbol{\phi}_{h}\|_0.  
\end{align}	
\end{lemma}
\begin{proof}
Let $\hat{\boldsymbol{\phi}}=\nabla\hat{u}\in H_{0}^{1}(\Omega;\mathbb{R}^{3})\cap\ker(\curl)$, then \eqref{bdual1} can be written as
\begin{equation*}
-\varepsilon^2\div\Delta\hat{\boldsymbol{\phi}} +\div\hat{\boldsymbol{\phi}}=\div(\boldsymbol{\phi} - I_h^{\rm ND}\boldsymbol{\phi}_{h}),\quad \mathrm{i.e.},\;\; \div(\varepsilon^2\Delta\hat{\boldsymbol{\phi}}-\hat{\boldsymbol{\phi}} +\boldsymbol{\phi} - I_h^{\rm ND}\boldsymbol{\phi}_{h})=0.
\end{equation*}%
By Theorem 1.1 in \cite{CostabelMcIntosh2010}, there exists a $\hat{\boldsymbol{p}}\in H^1(\Omega; \mathbb{R}^{3})$ satisfying
\begin{equation*}
\curl \hat{\boldsymbol{p}} =\varepsilon^2\Delta\hat{\boldsymbol{\phi}}-\hat{\boldsymbol{\phi}} +\boldsymbol{\phi} - I_h^{\rm ND}\boldsymbol{\phi}_{h}.
\end{equation*}
Then we employ \eqref{uhatregularity} to end the proof. 
\end{proof}
\begin{lemma}
Let $(w, \boldsymbol{\phi}=\nabla u, \boldsymbol{p}, 0,u)$ be the solution of the decoupled formulation~\eqref{decoupleA}, and $(w_h, \boldsymbol{\phi}_h, \boldsymbol{p}_h, 0, u_h)$ be the solution of the discrete method~\eqref{ImprovedHcurl}. Assume the $H^3$ regularity \eqref{iharmonicdualregularity} holds. We have
\begin{equation}\label{eq:phihphestimateL2}
|u-u_h|_1=\|\boldsymbol{\phi}-I_h^{\rm ND}\boldsymbol{\phi}_{h}\|_{0}\lesssim \varepsilon^{-3/2}h^2 \|f\|_{0}.
\end{equation}
\end{lemma}
\begin{proof}
By combining equation \eqref{bdual2}, the fact $\curl(\boldsymbol{\phi} - I_h^{\rm ND}\boldsymbol{\phi}_{h})=0$, and integration by parts, we obtain
\begin{align}\label{wqg2}
\|\boldsymbol{\phi} - I_h^{\rm ND}\boldsymbol{\phi}_{h}\|_0^{2}
&=(\boldsymbol{\phi} - I_h^{\rm ND}\boldsymbol{\phi}_{h},
-\varepsilon^2\Delta\hat{\boldsymbol{\phi}}+\hat{\boldsymbol{\phi}}+\curl \hat{\boldsymbol{p}})\notag  \\
&= - \varepsilon^2(\boldsymbol{\phi} - I_h^{\rm ND}\boldsymbol{\phi}_{h},\Delta\hat{\boldsymbol{\phi}})
+ (\boldsymbol{\phi} - I_h^{\rm ND}\boldsymbol{\phi}_{h},\hat{\boldsymbol{\phi}}) \notag  \\
&= -\varepsilon^2(\boldsymbol{\phi}_h-I_h^{\rm ND}\boldsymbol{\phi}_h,\Delta\hat{\boldsymbol{\phi}}) - \varepsilon^2(\boldsymbol{\phi}-\boldsymbol{\phi}_h,\Delta\hat{\boldsymbol{\phi}}) + (\boldsymbol{\phi} - I_h^{\rm ND}\boldsymbol{\phi}_{h},\hat{\boldsymbol{\phi}})\\
&=: I_1+I_2+I_3,\notag 
\end{align} 
where
\begin{align*}
I_1&=-\varepsilon^2(\boldsymbol{\phi}_h-I_h^{\rm ND}\boldsymbol{\phi}_h,\Delta\hat{\boldsymbol{\phi}})-\varepsilon^2\sum_{T\in\mathcal T_h}(\partial_n\hat{\boldsymbol{\phi}}, \boldsymbol{\phi}-\boldsymbol{\phi}_h)_{\partial T},\\
I_2&= (\boldsymbol{\phi} - I_h^{\rm ND}\boldsymbol{\phi}_{h},\hat{\boldsymbol{\phi}}-\Pi_h^{\rm ND}\hat{\boldsymbol{\phi}}) + \varepsilon^2(\nabla_h(\boldsymbol{\phi}-\boldsymbol{\phi}_h),\nabla_h(\hat{\boldsymbol{\phi}}-I_h^{\Phi}\hat{\boldsymbol{\phi}})), \\
I_3&=(\boldsymbol{\phi} - I_h^{\rm ND}\boldsymbol{\phi}_{h},\Pi_h^{\rm ND}\hat{\boldsymbol{\phi}}) + \varepsilon^2(\nabla_h(\boldsymbol{\phi}-\boldsymbol{\phi}_h),\nabla_h(I_h^{\Phi}\hat{\boldsymbol{\phi}})).
\end{align*}

Next, we estimate $I_1$, $I_2$ and $I_3$ in \eqref{wqg2}.
By \eqref{eq:IhNDestimes}, \eqref{varepsilonIh2} with $r=0$ and $s=2$, \eqref{eq:IhND2} and \eqref{sregularitygrad},
\begin{align*}
\|\boldsymbol{\phi}_h - I_h^{\rm ND}\boldsymbol{\phi}_{h}\|_0&=\|\boldsymbol{\phi}_h-\Pi_h^{\rm ND}\boldsymbol{\phi} - I_h^{\rm ND}(\boldsymbol{\phi}_{h}-\Pi_h^{\rm ND}\boldsymbol{\phi})\|_0\lesssim h|\boldsymbol{\phi}_h-\Pi_h^{\rm ND}\boldsymbol{\phi}|_{1,h} \\
& \leq h|\boldsymbol{\phi}-\boldsymbol{\phi}_h|_{1,h} + h|\boldsymbol{\phi}-\Pi_h^{\rm ND}\boldsymbol{\phi}|_{1,h} \lesssim \varepsilon^{-3/2}h^2 \|f\|_{0}.
\end{align*}
This together with \eqref{bdual2regularity} and the weak continuity \eqref{Phiweakcontinuity1} implies
\begin{align}\label{wqg20}
I_1&\lesssim   \varepsilon^{1/2} h^2\|f\|_{0}|\hat{\boldsymbol{\phi}}|_2+\varepsilon^2h|\boldsymbol{\phi}-\boldsymbol{\phi}_h|_{1,h}|\hat{\boldsymbol{\phi}}|_2 \lesssim  \varepsilon^{-3/2}h^2 \|f\|_{0}\|\boldsymbol{\phi} - I_h^{\rm ND}\boldsymbol{\phi}_{h}\|_0.
\end{align} 
Applying \eqref{varepsilonIh2} with $r=0$ and $s=2$, \eqref{eq:IhND2}, \eqref{BEqs} and \eqref{bdual2regularity},
\begin{align}\label{wqg22}
I_2&\lesssim \|\boldsymbol{\phi} - I_h^{\rm ND}\boldsymbol{\phi}_{h}\|_0\|\hat{\boldsymbol{\phi}}-\Pi_h^{\rm ND}\hat{\boldsymbol{\phi}}\|_0 + \varepsilon^2|\boldsymbol{\phi}-\boldsymbol{\phi}_h|_{1,h}|\hat{\boldsymbol{\phi}}-I_h^{\Phi}\hat{\boldsymbol{\phi}}|_{1,h} \\
\notag
&\lesssim  \varepsilon^{-1/2}h^2 \|f\|_{0}|\hat{\boldsymbol{\phi}}|_1 + \varepsilon^{1/2}h^2 \|f\|_{0}|\hat{\boldsymbol{\phi}}|_2 \lesssim  \varepsilon^{-3/2}h^2 \|f\|_{0}\|\boldsymbol{\phi} - I_h^{\rm ND}\boldsymbol{\phi}_{h}\|_0.
\end{align}
By the commutative diagram \eqref{dcommutesdeRham}, $\Pi_h^{\rm ND}\hat{\boldsymbol{\phi}}=\nabla\Pi_h^{\grad}\hat{u}$. Then we acquire from equation \eqref{CC} with $\boldsymbol{\psi}=I_h^{\Phi}\hat{\boldsymbol{\phi}}$ and the Galerkin orthogonality \eqref{orthogonality} that
\begin{align*}
I_3&=(\boldsymbol{\phi},\Pi_h^{\rm ND}\hat{\boldsymbol{\phi}}) + \varepsilon^2(\nabla\boldsymbol{\phi},\nabla_h(I_h^{\Phi}\hat{\boldsymbol{\phi}})) -  (\nabla w_h,\Pi_h^{\rm ND}\hat{\boldsymbol{\phi}}) \\
&=(\boldsymbol{\phi},\Pi_h^{\rm ND}\hat{\boldsymbol{\phi}}) + \varepsilon^2(\nabla\boldsymbol{\phi},\nabla_h(I_h^{\Phi}\hat{\boldsymbol{\phi}})) -  (\nabla w,\Pi_h^{\rm ND}\hat{\boldsymbol{\phi}}).
\end{align*}
Noting the fact $\nabla w= - \varepsilon^2\Delta\boldsymbol{\phi}+\boldsymbol{\phi}+\curl\boldsymbol{p}$ by \eqref{20250603}, we have from the weak continuity \eqref{Phiweakcontinuity1}, estimate \eqref{eq:IhND2} of $\Pi_h^{\rm ND}$, estimate \eqref{BEqs} of $I_h^{\Phi}$, and the regularity \eqref{sregularitygrad} and \eqref{bdual2regularity} that
\begin{align*}
I_3&=\varepsilon^2(\nabla\boldsymbol{\phi},\nabla_h(I_h^{\Phi}\hat{\boldsymbol{\phi}})) + \varepsilon^2(\Delta\boldsymbol{\phi},\Pi_h^{\rm ND}\hat{\boldsymbol{\phi}}) \\
&=\varepsilon^2(\Delta\boldsymbol{\phi},\Pi_h^{\rm ND}\hat{\boldsymbol{\phi}}-I_h^{\Phi}\hat{\boldsymbol{\phi}}) + \varepsilon^2\sum_{T\in\mathcal T_h}(\partial_n\boldsymbol{\phi}, I_h^{\Phi}\hat{\boldsymbol{\phi}}-\hat{\boldsymbol{\phi}})_{\partial T} \\
&\lesssim \varepsilon^2|\boldsymbol{\phi}|_2(\|\Pi_h^{\rm ND}\hat{\boldsymbol{\phi}}-I_h^{\Phi}\hat{\boldsymbol{\phi}}\|_0+h|I_h^{\Phi}\hat{\boldsymbol{\phi}}-\hat{\boldsymbol{\phi}}|_{1,h}) \\
&\lesssim \varepsilon^2h^2|\boldsymbol{\phi}|_2|\hat{\boldsymbol{\phi}}|_2\lesssim  \varepsilon^{-3/2}h^2 \|f\|_{0}\|\boldsymbol{\phi} - I_h^{\rm ND}\boldsymbol{\phi}_{h}\|_0.
\end{align*}
Finally, substituting the estimates \eqref{wqg20}, \eqref{wqg22}, and the above into \eqref{wqg2}, we obtain the desired resulf \eqref{eq:phihphestimateL2}.
\end{proof}

Let $\tilde{u}\in H_0^1(\Omega)$ be the solution to the following dual problem:
\begin{equation}
-\Delta \tilde{u}=u-u_h.   \label{eqdual}
\end{equation}
Assume that the dual problem \eqref{eqdual} admits the following regularity estimate:
\begin{equation}\label{utilderegularity}
\|\tilde{u}\|_2\lesssim \|u-u_{h}\|_0.
\end{equation}
We refer the reader to \cite{MitreaMitreaYan2010, GaoLai2020, Kadlec1964, Talenti1965} for regularity results pertaining to the Poisson equation on polytopal domains.

\begin{lemma}
Let $(w, \boldsymbol{\phi}=\nabla u, \boldsymbol{p},0,u)$ be the solution of the decoupled formulation~\eqref{decoupleA}, and $(w_h, \boldsymbol{\phi}_h, \boldsymbol{p}_h, 0, u_h)$ be the solution of the discrete method~\eqref{ImprovedHcurl}. Assume the regularity \eqref{utilderegularity} holds. Assume $u_0\in H^s(\Omega)$ with $2\leq s\leq 3$. We have
\begin{align}
\|u-u_{h}\|_{0}&\lesssim \varepsilon^{1/2}h \|f\|_{0} +h^{s}|u_0|_{s},\label{uL2error1}\\
\|u_0-u_{h}\|_{0}&\lesssim  \varepsilon^{1/2}  \|f\|_{0} +h^{s}|u_0|_{s}.\label{u0L2error}
\end{align}
\end{lemma}
\begin{proof}
Using integration by parts on \eqref{eqdual}, $\nabla(\Pi_h^{\grad}\tilde{u})=\Pi_h^{\rm ND}(\nabla\tilde{u})$ by the commutative diagram~\eqref{dcommutesdeRham}, and estimate \eqref{eq:IhND2} of $\Pi_h^{\rm ND}$, we obtain
\begin{equation}\label{utilde}
\begin{aligned}
\|u-u_h\|_0^{2}
&=-(u-u_h,\Delta \tilde{u})=(\nabla(u-u_h),\nabla\tilde{u}) \\
&=(\nabla(u-u_h),\nabla(\tilde{u}-\Pi_h^{\grad}\tilde{u})) +(\boldsymbol{\phi}-I_h^{\rm ND}\boldsymbol{\phi}_{h},\nabla(\Pi_h^{\grad}\tilde{u}))\\
&=(\nabla(u-u_h),\nabla\tilde{u}-\Pi_h^{\rm ND}(\nabla\tilde{u})) +(\boldsymbol{\phi}-I_h^{\rm ND}\boldsymbol{\phi}_{h},\Pi_h^{\rm ND}(\nabla\tilde{u}))\\
&\lesssim\,  h|u-u_h|_1|\tilde{u}|_2+(\boldsymbol{\phi}-I_h^{\rm ND}\boldsymbol{\phi}_{h},\Pi_h^{\rm ND}(\nabla\tilde{u})).
\end{aligned} 
\end{equation}

Notice that $\curl(\Pi_h^{\rm ND}(\nabla\tilde{u}))=0$.
By taking $\boldsymbol{\psi} = I_h^{\Phi}(\nabla\tilde{u})$ and $\mu = 0$ in equation~\eqref{CC}, we get
\begin{equation*}
\varepsilon^2(\nabla_h\boldsymbol{\phi}_h,\nabla_{h}(I_h^{\Phi}(\nabla\tilde{u}))) + (I_h^{\rm ND}\boldsymbol{\phi}_h,\Pi_h^{\rm ND}(\nabla\tilde{u})) = (\nabla w_h,\Pi_h^{\rm ND}(\nabla\tilde{u})).
\end{equation*}
This combined with the Galerkin orthogonality condition \eqref{orthogonality} induces
\begin{align*}
(\boldsymbol{\phi}-I_h^{\rm ND}\boldsymbol{\phi}_{h},\Pi_h^{\rm ND}(\nabla\tilde{u}))
&=(\boldsymbol{\phi},\Pi_h^{\rm ND}(\nabla\tilde{u}))-(I_h^{\rm ND}\boldsymbol{\phi}_{h},\Pi_h^{\rm ND}(\nabla\tilde{u}))\\
&=(\boldsymbol{\phi},\Pi_h^{\rm ND}(\nabla\tilde{u}))-(\nabla w_h,\Pi_h^{\rm ND}(\nabla\tilde{u})) \\
&\quad+\varepsilon^2(\nabla_h\boldsymbol{\phi}_h,\nabla_{h}(I_h^{\Phi}(\nabla\tilde{u})))\\
&=(\boldsymbol{\phi}-\nabla w,\Pi_h^{\rm ND}(\nabla\tilde{u}))+\varepsilon^2(\nabla_h\boldsymbol{\phi}_h,\nabla_{h}(I_h^{\Phi}(\nabla\tilde{u}))).
\end{align*} 

Using the identity $\boldsymbol{\phi} = \nabla w + \varepsilon^2\Delta \boldsymbol{\phi} - \curl \boldsymbol{p}$, together with the integration by parts, 
we obtain
\begin{equation*}
(\boldsymbol{\phi}-I_h^{\rm ND}\boldsymbol{\phi}_{h},\Pi_h^{\rm ND}(\nabla\tilde{u})) =(  \varepsilon^2\Delta\boldsymbol{\phi},\Pi_h^{\rm ND}(\nabla\tilde{u}))+\varepsilon ^2(\nabla_h \boldsymbol{\phi}_h,\nabla_h I_h^{\Phi}(\nabla\tilde{u}))
=: I_1+I_2,
\end{equation*}
where
\begin{align*}
I_1&=( \varepsilon^2\Delta\boldsymbol{\phi},\Pi_h^{\rm ND}(\nabla\tilde{u})-I_h^{\Phi}(\nabla\tilde{u})),\\
I_2&= ( \varepsilon^2\Delta\boldsymbol{\phi},I_h^{\Phi}(\nabla\tilde{u}))
	+\varepsilon ^2(\nabla_h \boldsymbol{\phi}_h,\nabla_h I_h^{\Phi}(\nabla\tilde{u})).
\end{align*}
Now let us estimate the terms $I_1$ and $I_2$.
By employing estimate \eqref{eq:IhND2} of $\Pi_h^{\rm ND}$, estimate \eqref{BEqs} of $I_h^{\Phi}$, together with the regularity results \eqref{sregularitygrad} and \eqref{utilderegularity}, we derive the following estimate:
\begin{equation*}
I_1\lesssim  \varepsilon^2|\boldsymbol{\phi}|_2\|\Pi_h^{\rm ND}(\nabla\tilde{u})-I_h^{\Phi}(\nabla\tilde{u})\|_0\lesssim  \varepsilon^2 h |\boldsymbol{\phi}|_2|\tilde{u}|_2\lesssim   \varepsilon^{1/2}h \|f\|_{0}\|u-u_h\|_0.
\end{equation*}
Applying the weak continuity \eqref{Phiweakcontinuity1}, estimate \eqref{BEqs} of $I_h^{\Phi}$, estimate \eqref{varepsilonIh2} with $r=0$ and $s=2$, the regularity \eqref{sregularitygrad} and the regularity \eqref{utilderegularity}, then
\begin{align*}
I_2&=\varepsilon^2 ( \nabla_h (\boldsymbol{\phi}_h-\boldsymbol{\phi}),\nabla_h I_h^{\Phi}(\nabla\tilde{u}))
+\varepsilon^2\sum_{T\in\mathcal T_h}(\partial_n\boldsymbol{\phi}, I_h^{\Phi}(\nabla\tilde{u}))_{\partial T}\\
&\lesssim\varepsilon^2(|\boldsymbol{\phi}-\boldsymbol{\phi}_h |_{1,h}+ h|\boldsymbol{\phi}|_2|)|I_h^{\Phi}(\nabla\tilde{u})|_{1,h}\lesssim\varepsilon^2(|\boldsymbol{\phi}-\boldsymbol{\phi}_h |_{1,h}+ h|\boldsymbol{\phi}|_2|)|\tilde{u}|_2\\
&\lesssim(\varepsilon h |u_0|_2 +\varepsilon^{1/2}h \|f\|_{0})\|u-u_h\|_0.
\end{align*}
Then by the regularity \eqref{H2regularity}.
\begin{align*}
(\boldsymbol{\phi}-I_h^{\rm ND}\boldsymbol{\phi}_{h},\nabla\Pi_h^{\grad}\tilde{u})
\lesssim \varepsilon^{1/2}h \|f\|_{0}\|u-u_h\|_0.
\end{align*}
Combining this with estimate \eqref{eq:errorestimatauhH03} for the case $r = 1$, and equation \eqref{utilde}, the estimate \eqref{uL2error1} is established. Moreover, since $\|u_0-u_{h}\|_{0}\lesssim\|u_0-u\|_{0}+\|u-u_{h}\|_{0}$, 
applying the estimate \eqref{uL2error1} together with the regularity result \eqref{sregularity}, we arrive at the desired estimate \eqref{u0L2error}.
\end{proof}

\begin{remark}\rm
By selecting $r=0$,
the estimates \eqref{varepsilonIh2} and \eqref{eq:phihphestimateL2} show that 
\begin{equation*}
\varepsilon|\boldsymbol{\phi}-\boldsymbol{\phi}_{h}|_{1,h}+\|\boldsymbol{\phi}-I_h^{\rm ND} \boldsymbol{\phi}_h\|_0=\mathcal{O}(h),\quad |u-u_h|_1=\mathcal{O}(h^2),\quad \|u-u_h\|_0=\mathcal{O}(h^2)
\end{equation*}
for fixed parameter $\varepsilon>0$.
When $\varepsilon \rightarrow 0$ and $u_0\in H^3(\Omega)$, 
the estimates \eqref{varepsilonIh2}-\eqref{eq:errorestimatauhH04} with $r=1$ and \eqref{uL2error1}-\eqref{u0L2error} show that 
\begin{equation*}
\varepsilon|\boldsymbol{\phi}-\boldsymbol{\phi}_{h}|_{1,h}+\|\boldsymbol{\phi}-I_h^{\rm ND} \boldsymbol{\phi}_h\|_0=\mathcal{O}(h^2),\quad |u-u_h|_1=\mathcal{O}(h^2), \quad \|u-u_h\|_0=\mathcal{O}(h^3)
\end{equation*}
\begin{equation*}
\varepsilon|\boldsymbol{\phi}_0-\boldsymbol{\phi}_{h}|_{1,h}+\|\boldsymbol{\phi}_0-I_h^{\rm ND} \boldsymbol{\phi}_h\|_0=\mathcal{O}(h^2),\quad |u_0-u_h|_1=\mathcal{O}(h^2),\quad \|u_0-u_h\|_0=\mathcal{O}(h^3).
\end{equation*}
Therefore, estimates \eqref{varepsilonIh2}-\eqref{eq:errorestimatauhH04}, \eqref{eq:phihphestimateL2} and \eqref{uL2error1}-\eqref{u0L2error} are optimal and robust with respect to both the parameter  $\varepsilon$ and the mesh size $h$.
\end{remark}

\begin{remark}\rm\label{nointerpolation}
	An alternative decoupled finite element method for the fourth-order singular perturbation problem \eqref{eqtwo}, which avoids the use of the interpolation operator $I_h^{\rm ND}$, is defined as follows: find $w_h \in V_h^{\grad}$, $\boldsymbol{\phi}_{h0} \in \Phi_h$, $\boldsymbol{p}_{h0} \in V_h^{\div}$, $\lambda_{h0} \in \mathcal{Q}_h$, and $u_{h0} \in V_h^{\grad}$ such that
	\begin{subequations}\label{decoupleHcurlgrad}
	\begin{align}
				(\nabla w_h,\nabla v)&=(f,v),\label{BA}\\
				\varepsilon^2\tilde{a}_h(\boldsymbol{\phi}_{h0},\lambda_{h0};\boldsymbol{\psi},\mu)
				+(\curl\boldsymbol{\psi},\boldsymbol{p}_{h0})&=(\nabla w_{h},\boldsymbol{\psi}),\label{BC} \\
				(\mu,\div\boldsymbol{p}_{h0})&=0,\label{BC0} \\
				b(\boldsymbol{\phi}_{h0}, \lambda_{h0}; \boldsymbol{q}) &=0 ,\label{BD}\\
				(\nabla u_{h0},\nabla \chi)&=(\boldsymbol{\phi}_{h0},\nabla\chi),\label{BE}
		\end{align}
		for any $v,\,\chi\in V_{h}^{\grad}$, $\boldsymbol{\psi}\in \Phi_{h},\,\mu\in\mathcal{Q}_{h}$, and $\boldsymbol{q}\in  V_{h}^{\div}$,
	\end{subequations}
 where $\tilde{a}_h(\boldsymbol{\phi}_{h0},\lambda_{h0};\boldsymbol{\psi},\mu):=\varepsilon^2(\nabla_h\boldsymbol{\phi}_{h0},\nabla_h\boldsymbol{\psi})+
		(\boldsymbol{\phi}_{h0},\boldsymbol{\psi})$. 
	Following the analysis in this section, the method \eqref{decoupleHcurlgrad} is well-posed, and it satisfies $\curl \boldsymbol{\phi}_{h0} = 0$ and $\lambda_{h0} = 0$. As $\varepsilon \to 0$, we can establish the following error estimates:
	\begin{align*}
		\varepsilon|\boldsymbol{\phi}-\boldsymbol{\phi}_{h}|_{1,h}+\|\boldsymbol{\phi}- \boldsymbol{\phi}_h\|_0&=\mathcal{O}(h^{1/2}),\\
		\varepsilon|\boldsymbol{\phi}_0-\boldsymbol{\phi}_{h}|_{1,h}+\|\boldsymbol{\phi}_0-\boldsymbol{\phi}_h\|_0 &=\mathcal{O}(h^{1/2}).
	\end{align*}
	These estimates demonstrate a sharp, though suboptimal, half-order convergence rate, consistent with the results reported in~\cite[Theorem 3.4]{HuangShiWang2021} and \cite[Theorem 3.6]{WangHuangTangZhou2018}. This half-order convergence rate is further validated by numerical experiments. The motivation for introducing the interpolation operator $I_h^{\rm ND}$ in the decoupled finite element method \eqref{ImprovedHcurl} lies in the desire to achieve the optimal convergence.
\end{remark}

\section{Numerical results}\label{chap6}
This section presents numerical results to validate the proposed decoupled mixed finite element method. 
Let $\Omega$ be the unit cube domain $(0,1)^3$.
Introduce error notation
\begin{align*}
	{\rm Err}(\boldsymbol{\phi})&:=(\varepsilon^2|\boldsymbol{\phi}-\boldsymbol{\phi}_{h}|_{1,h}^2+\|\boldsymbol{\phi}-I_h^{\rm ND} \boldsymbol{\phi}_h\|_0^2)^{1/2},\\
	{\rm Err}_{0}(\boldsymbol{\phi})&:=(\varepsilon^2|\boldsymbol{\phi}-\boldsymbol{\phi}_{h0}|_{1,h}^2+\|\boldsymbol{\phi}-\boldsymbol{\phi}_{h0}\|_0^2)^{1/2}.
\end{align*}

\subsection{Numerical test without boundary layer}
We consider the fourth-order elliptic singular perturbation problem \eqref{eqtwo} with the exact solution given by
\begin{align*}
	u(\boldsymbol{x})=\sin^2(\pi x)\sin^2(\pi y)\sin^2(\pi z).
\end{align*}
The source term $f$ is computed from \eqref{eqtwo}.

The corresponding numerical errors are summarized in Table~\ref{errorIhPhi}, 
including 
${\rm Err}(\boldsymbol{\phi})$, $\|u-u_{h}\|_0$  and $|u-u_{h}|_1$ for various values of $\varepsilon$ and mesh size~$h$.
As shown in Table~\ref{errorIhPhi}, ${\rm Err}(\boldsymbol{\phi})$ exhibits the first-order convergence $\mathcal{O}(h)$ for $\varepsilon = 1$ and $10^{-1}$, consistent with the theoretical estimate~\eqref{varepsilonIh2} for $r = 0$ and $s=2$. For $\varepsilon = 10^{-4}$ and $10^{-6}$, the convergence improves to second order $\mathcal{O}(h^2)$, which agrees with the theoretical result~\eqref{varepsilonIh2} for $r = 1$ and $s=3$.
Table~\ref{errorIhPhi} shows that $|u-u_{h}|_1=\mathcal{O}(h^2)$ for $\varepsilon = 1,\,10^{-1}$, which agrees with the theoretical result~\eqref{eq:phihphestimateL2}. For $\varepsilon = 10^{-4}$ and $10^{-6}$, the convergence rate remains $\mathcal{O}(h^2)$, consistent with the theoretical estimate~\eqref{eq:errorestimatauhH03} for $r = 1$ and $s=3$. Furthermore, the $L^2$-norm error $\|u - u_h\|_0$ exhibits second-order convergence, i.e., $\mathcal{O}(h^2)$, for $\varepsilon = 1$ and $10^{-1}$, which is consistent with the theoretical estimate~\eqref{eq:phihphestimateL2}. As $\varepsilon$ decreases to $10^{-4}$ and $10^{-6}$, the convergence improves to third order, i.e., $\mathcal{O}(h^3)$, in agreement with the refined estimate~\eqref{uL2error1}.
\begin{table}[ht]
	\centering
	\caption{Errors\,\,
		${\rm Err}(\boldsymbol{\phi})$,\,\,$\|u-u_{h}\|_0$ and $|u-u_{h}|_1$ of the decoupled method \eqref{ImprovedHcurl} without boundary layer.}
	\label{errorIhPhi}
	\begin{tabular}{cccc|cc|cccc}
		\toprule
		$\varepsilon$ \,  &  $h$
		  & ${\rm Err}(\boldsymbol{\phi})$ \,  &  rate   \,  	&$\|u-u_{h}\|_0$ \,  &  rate     \,  &  $|u-u_{h}|_1$ \,  &  rate\\
		\midrule
		\multirow{5}{*}{$1$} 
		& $2^{-2}$   &7.862e+00    &  -          &1.098e-01       &  -& 6.965e-01 &  -&  \\
		& $2^{-3}$    &4.924e+00  & 0.68     &4.378e-02       &  1.33&  2.887e-01& 1.27 &    \\
		& $2^{-4}$   &2.776e+00   & 0.83     & 1.429e-02      &1.62&  9.735e-02& 1.57 &   \\
		& $2^{-5}$   &1.466e+00   & 0.92     & 4.101e-03        & 1.80& 2.857e-02& 1.77&   \\
		\midrule
		\multirow{5}{*}{$10^{-1}$} 
		& $2^{-2}$    &9.538e-01    & -            & 6.889e-02     & -          &4.656e-01        &  -&   \\
		& $2^{-3}$    &5.309e-01    & 0.85     & 2.175e-02       &1.66    &1.610e-01          &1.53&\\
		& $2^{-4}$    &2.842e-01    & 0.90     & 6.373e-03    &1.77     &4.913e-02         &1.71&  \\
		& $2^{-5}$    &1.476e-01     & 0.95     & 1.772e-03      &1.85     &1.395e-02        &1.82&    \\
		\midrule
		\multirow{5}{*}{$10^{-4}$} 
		& $2^{-2}$    &2.572e-01     &  -         & 9.082e-03    &  -        &1.985e-01        & -&  \\
		& $2^{-3}$   &7.295e-02     & 1.82     &1.118e-03      &3.02     &5.645e-02      &1.81& \\
		& $2^{-4}$   &1.910e-02      & 1.93     &1.376e-04     &3.02     &1.483e-02      &1.93&  \\
		& $2^{-5}$   &4.838e-03     & 1.98     & 1.711e-05     &3.01     &3.765e-03      &1.98&  \\
		\midrule
		\multirow{5}{*}{$10^{-6}$} 
		& $2^{-2}$   &2.572e-01     & -           &9.082e-03   &  -           &1.985e-01      & -&  \\
		& $2^{-3}$   &7.295e-02    &1.82       &1.118e-03     &3.02      &5.645e-02     &1.81& \\
		& $2^{-4}$   &1.910e-02    &1.93        &1.376e-04    &3.02      &1.483e-02      &1.93&  \\
		& $2^{-5}$   &4.838e-03   &1.98       &1.711e-05      &3.01      &3.765e-03      &1.98&  \\
		\bottomrule
	\end{tabular}
\end{table}


\subsection{Numerical test with boundary layer}
We next verify the convergence behavior of the nonconforming finite element methods~\eqref{ImprovedHcurl} and~\eqref{decoupleHcurlgrad} for problem~\eqref{eqtwo} with boundary layers.

To this end, consider the exact solution to the Poisson equation~\eqref{eqtwo1}:
\begin{align*}
	u_0(\boldsymbol{x})=\sin(\pi x)\sin(\pi y)\sin(\pi z).
\end{align*}
The corresponding right-hand side $f$ computed from~\eqref{eqtwo1} is used in~\eqref{eqtwo}, for which the exact solution $u$ is not known in closed form. When $\varepsilon$ is small, the solution $u$ develops pronounced boundary layers. In the numerical experiments that follow, we set $\varepsilon=10^{-6},\,10^{-8},\,10^{-10}$ to investigate such behavior.

Numerical results for the decoupled method~\eqref{decoupleHcurlgrad} are presented in Table~\ref{errorPhiP01}, which reports  ${\rm Err}_{0}(\boldsymbol{\phi}_0)$, $\|u_0-u_{h0}\|_0$ and $|u_0-u_{h0}|_1$. From Table~\ref{errorPhiP01}, we observe that ${\rm Err}_{0}(\boldsymbol{\phi}_0)=\mathcal{O}(h^{1/2})$, which agrees with the theoretical estimate outlined in Remark~\ref{nointerpolation}. In addition, Table~\ref{errorPhiP01} shows that
$|u_0-u_{h0}|_1=\mathcal{O}(h^{1/2})$, as anticipated, while $\|u_0-u_{h0}\|_0=\mathcal{O}(h)$.
\begin{table}[ht]
	\centering
	\caption{Errors\,\,${\rm Err}_{0}(\boldsymbol{\phi}_0)$,\,\,$\|u_0-u_{h0}\|_0$ and $|u_0-u_{h0}|_1$of the decoupled method \eqref{decoupleHcurlgrad} with boundary layer.}
	\label{errorPhiP01}
	\begin{tabular}{cccc|cc|cccc}
		\toprule
		$\varepsilon$ \,  &  	$h$
		  & ${\rm Err}_{0}(\boldsymbol{\phi}_0)$ \,  &  rate   \,&$\|u_0-u_{h0}\|_0$ \,  &  rate     \,  &  $|u_0-u_{h0}|_1$ \,  &  rate\\
		\midrule
		\multirow{5}{*}{$10^{-6}$} 
		& $2^{-2}$  &    5.475e-01 &    -    & 4.944e-02 &  -&  4.183e-01&  -&  \\
		& $2^{-3}$  &   3.643e-01 & 0.59 &  2.278e-02& 1.12& 2.534e-01&  0.72&    \\
		& $2^{-4}$  &   2.527e-01  & 0.53 & 1.103e-02  & 1.04& 1.656e-01&  0.61&   \\
		& $2^{-5}$ &    1.777e-01  & 0.51 &5.444e-03 &  1.02&  1.127e-01&0.56&  \\
		\midrule 
		\multirow{5}{*}{$10^{-8}$} 
			& $2^{-2}$  &    5.475e-01 &    -    & 4.944e-02 &  -&  4.183e-01&  -&  \\
		& $2^{-3}$  &   3.643e-01 & 0.59 &  2.278e-02& 1.12& 2.534e-01&  0.72&    \\
		& $2^{-4}$  &   2.527e-01  & 0.53 & 1.103e-02  & 1.04& 1.656e-01&  0.61&   \\
		& $2^{-5}$ &    1.777e-01  & 0.51 &5.444e-03 &  1.02&  1.127e-01&0.56&  \\
		\midrule
		\multirow{5}{*}{$10^{-10}$} 
		& $2^{-2}$  &    5.475e-01 &    -    & 4.944e-02 &  -&  4.183e-01&  -&  \\
	& $2^{-3}$  &   3.643e-01 & 0.59 &  2.278e-02& 1.12& 2.534e-01&  0.72&    \\
	& $2^{-4}$  &   2.527e-01  & 0.53 & 1.103e-02  & 1.04& 1.656e-01&  0.61&   \\
	& $2^{-5}$ &    1.777e-01  & 0.51 &5.444e-03 &  1.02&  1.127e-01&0.56&  \\
		\bottomrule
	\end{tabular}
\end{table}

Table~\ref{optimalerrorPhiP} displays the numerical results for the decoupled method~\eqref{ImprovedHcurl}, including ${\rm Err}(\boldsymbol{\phi}_0)$, $\|u_0-u_{h}\|_0$ and $|u_0-u_{h}|_1$. As evidenced in Table~\ref{optimalerrorPhiP}, the error ${\rm Err}(\boldsymbol{\phi}_0)=\mathcal{O}(h^2)$,
confirming the theoretical estimate in~\eqref{varepsilonIh3} for $s=3$. Additionally, Table~\ref{optimalerrorPhiP} indicates that
$|u_0-u_{h}|_1=\mathcal{O}(h^{2})$, aligning with the theoretical result~\eqref{eq:errorestimatauhH04} for $s=3$. Moreover, the $L^2$-norm error $\|u_0 - u_h\|_0$ achieves optimal third-order convergence, i.e., $\mathcal{O}(h^3)$, in agreement with the refined estimate~\eqref{u0L2error}.

In conclusion, when a boundary layer phenomenon is present, the decoupled method \eqref{ImprovedHcurl} significantly outperforms the decoupled method \eqref{decoupleHcurlgrad} by achieving higher-order convergence rates.
\begin{table}[ht]
	\centering
	\caption{Errors\,\,
		${\rm Err}(\boldsymbol{\phi}_0)$,\,\,$\|u_0-u_{h}\|_0$ and $|u_0-u_{h}|_1$ of the decoupled method \eqref{ImprovedHcurl} with boundary layer.}
	\label{optimalerrorPhiP}
	\begin{tabular}{cccc|cc|cccc}
		\toprule
		$\varepsilon$ \,  &  $h$
		&$	{\rm Err}(\boldsymbol{\phi}_0)$ \,  &  rate  \, &$\|u_0-u_{h}\|_0$ \,  &  rate     \,  &  $|u_0-u_{h}|_1$ \,  &  rate\\
		\midrule
		\multirow{5}{*}{$10^{-6}$} 
		& $2^{-2}$   &1.692e-01      & -         &4.981e-03    &  2.98   &1.332e-01    &  -&\\
		& $2^{-3}$   &4.499e-02    & 1.91    & 6.038e-04   & 3.04    &3.556e-02   &1.91&\\
		& $2^{-4}$   & 1.148e-02    &1.97     &7.453e-05     &3.02    &9.089e-03   & 1.97 & \\
		& $2^{-5}$   & 2.888e-03   &1.99    &9.287e-06     &  3.00   &2.885e-03   &1.99&  \\
		\midrule
		\multirow{5}{*}{$10^{-8}$} 
		& $2^{-2}$   &1.692e-01       & -        &4.981e-03    &  2.98      &1.332e-01     &  -&\\
		& $2^{-3}$    &4.499e-02    & 1.91    & 6.038e-04  & 3.04       & 3.556e-02   &1.91&\\
		& $2^{-4}$    & 1.148e-02     & 1.97    &7.453e-05   &3.02        &9.089e-03    &1.97 & \\
		& $2^{-5}$   & 2.888e-03    & 1.99    &9.287e-06   &3.00        &2.885e-03    &1.99&  \\
		\midrule
		\multirow{5}{*}{$10^{-10}$} 
		& $2^{-2}$   &1.692e-01     & -          &4.981e-03   &  2.98&1.332e-01&  -&\\
		& $2^{-3}$   &4.499e-02   & 1.91     & 6.038e-04  & 3.04& 3.556e-02& 1.91&\\
		& $2^{-4}$   & 1.148e-02   & 1.97      &7.453e-05  &3.02&  9.089e-03& 1.97 & \\
		& $2^{-5}$   & 2.888e-03   &  1.99    &9.287e-06  &  3.00&2.885e-03&1.99&  \\
		\bottomrule
	\end{tabular}
\end{table}

\bibliographystyle{abbrv}
\bibliography{./refs}
\end{document}